\documentclass[11pt,a4paper]{amsart}
\usepackage{amsfonts}
\usepackage{amsthm}
\usepackage{amsmath}
\usepackage{amssymb}
\usepackage{amscd}
\usepackage[latin2]{inputenc}
\usepackage{t1enc}
\usepackage[mathscr]{eucal}
\usepackage{indentfirst}
\usepackage{graphicx}
\usepackage{graphics}
\usepackage{pict2e}
\usepackage{epic}
\numberwithin{equation}{section}
\usepackage[margin=3.5cm]{geometry}
\usepackage{epstopdf} 
\usepackage{tikz-cd}
\usepackage{mathtools}
\usepackage{cases}
\usepackage{hyperref}

\theoremstyle{plain}
\newtheorem{Th}{Theorem}[section]
\newtheorem{Lemma}[Th]{Lemma}
\newtheorem{Cor}[Th]{Corollary}
\newtheorem{Prop}[Th]{Proposition}

 \theoremstyle{definition}
\newtheorem{Def}[Th]{Definition}

\newtheorem{Rem}[Th]{Remark}
\newtheorem{?}[Th]{Question}

\newcommand{\RN}{\mathbb{R}}
\newcommand{\CN}{\mathbb{C}}

\newcommand{\ZN}{\mathbb{Z}}

\newcommand{\la}{\langle}
\newcommand{\ra}{\rangle}

\newcommand\norm[1]{\left\lVert#1\right\rVert}

\begin{document}

\title[Rarita-Schwinger-Seiberg-Witten equations]{$Pin(2)-$equivariance property of the Rarita-Schwinger-Seiberg-Witten equations}

\author[Minh Lam Nguyen]{Minh Lam Nguyen}

\address{Washington University in St. Louis \\ Department of Mathematics and Statistics \\
St. Louis MO 63130, U.S.A} 

\email{minhn@wustl.edu}



\begin{abstract}We define a variant of the Seiberg-Witten equations using the Rarita-Schwinger operators for a closed simply connected spin smooth $4-$manifold $X$. The moduli space of solutions to the system of non-linear differential equations consist of  harmonic $3/2-$spinors and $U(1)-$connections satisfying certain curvature condition. Beside having an obvious $U(1)-$symmetry, these equations also have a symmetry by $Pin(2)$. We exploit this additional symmetry to perform finite dimensional approximations for the eigenvalue problem of the $3/2-$monopole map and show that under a certain topological assumption, the moduli space of solutions is always non-compact, and thus non-empty. \end{abstract}

\maketitle

\tableofcontents

\section{Introduction} 
\subsection{Main result} The Rarita-Schwinger operator $Q$ was first introduced by Rarita and Schwinger in a 1941 paper to describe the wave functions of the so-called \textit{3/2-spinors} \cite{RS41}. It is a first-order elliptic operator that has the analytic continuation property. The higher spinors that live in the kernel of $Q$ are called \textit{Rarita-Schwinger fields}. In Physics, Rarita-Schwinger fields are important in the study of supergravity and superstring theory; thus, the literature in this area of science is vast. In Mathematics, this operator has not been studied as much except in the classical context of Clifford analysis.
 
Within geometry, Branson and Hijazi showed that $Q$ is conformally invariant. They also wrote down a Weitzenb\"ock-type formula for $Q^2$ \cite{BH02}. In contrast to the case of the spinor Dirac operator, the Weitzenb\"ock formula for $Q^2$ has a lower order term involving the Ricci curvature.  In a complicated situation where $Q$ is defined by a choice of a unitary $U(1)-$connection of some line bundle, it seems that the kernel of $Q$ should be studied under a restrictive geometric hypothesis of the underlying space. M. Wang also studied the role of the solutions to $Q$ in the deformation theory of Einstein metrics that admit parallel spinors \cite{W91}. These results could allude to the fact that $Q$ in the setting of a Seiberg-Witten-type gauge theory may have something new to say about smooth simply connected spin $4-$manifolds.

We briefly introduce the definition of a Rarita-Schwinger operator, more details will come in latter sections. Let $X$ be a closed spin smooth $4-$manifold and $\mathfrak{s}_{1/2}^{\pm}$ be the positive (negative) spinor bundle over $X$ associated to a certain spin structure. Any $spin^c$ structure on $X$ is given by $\mathfrak{s}_{1/2}\otimes_{\CN} L$, where $L$ is some line bundle. When $g$ is a fixed Riemannian metric on $X$, there is an associated Clifford multiplication $\rho: TX \otimes \mathfrak{s}^{\pm}_{1/2}\otimes L \to \mathfrak{s}^{\mp}_{1/2}\otimes L$. The positive (negative) $3/2-$spinor bundle $\mathfrak{s}^{\pm}_{3/2}\otimes L$ is  a sub-bundle of $TX\otimes \mathfrak{s}^{\pm}_{1/2} \otimes L$ given by $ker\, \rho$. Denote the orthogonal projection by $\pi^{\pm} : TX\otimes \mathfrak{s}^{\pm}_{1/2} \otimes L \to \mathfrak{s}^{\pm}_{3/2}\otimes L$. The twisted Rarita-Schwinger operator $Q^{\pm}_A : \Gamma(\mathfrak{s}^{\pm}_{3/2}\otimes L)\to \Gamma(\mathfrak{s}^{\mp}_{3/2}\otimes L)$ is defined as $Q^{\pm}_{A} = \pi^{\mp}\circ \mathcal{D}_{A}^{\pm}|_{\Gamma(\mathfrak{s}^{\pm}_{3/2}\otimes L)}$, where $\nabla_A$ is a unitary connection on $L$ and $\mathcal{D}_A$ is the Dirac operator associated to $TX \otimes \mathfrak{s}_{1/2}\otimes L$. A variant of the Seiberg-Witten equations (RS-SW) using the Rarita-Schwinger operator is given as follows
\begin{equation}
Q^+_A \psi = 0,\, \, \, \, F^+_A = \rho^{-1}(\mu(\psi)),
\end{equation}
where $\mu(\psi)$ is the traceless part of the endomorphism $\psi\psi^*$ on $\mathfrak{s}^+_{1/2}\otimes L$ and $F^+_A$ is the self-dual part of the curvature $F_A$. The unknowns of $(1.1)$ are pairs $(A,\psi)$ where $A$ is a $U(1)-$connection of $L$ and $\psi \in \Gamma(ker\, \rho)$.

\begin{?}
Are there non-trivial solutions to equations (1.1)?
\end{?}

The gauge symmetry of $(1.1)$ is given by $\mathcal{G} = Maps\,(X \to U(1))$. It is not hard to see that $(1.1)$ modulo gauge is a system of elliptic PDEs. If one wants to establish a gauge-theoretic invariant of $4-$manifolds by "counting" solutions to $(1.1)$ modulo gauge, one has to answer Question 1.1 first. It is not our intent to set up the full theory of RS-SW equations here. Rather, this paper, whose main purpose is to give some definitive answer to Question 1.1 in a certain context, serves as one of the important foundational steps to the full theory that we will address in our future work.  

The above question also could fit into the broader context of counting Rarita-Schwinger-fields considered by several people in recent years. Note that when $L=\underline{\CN}$ is the trivial line bundle, $Q_{A}$ above is reduced to the usual $untwisted$ Rarita-Schwinger operator $Q$ if $\nabla_A$ is just the trivial connection. In 2019, Homma and Semmelmann considered the problem of "counting solutions" to $Q$ \cite{HS19}. In their work, they gave a complete classification of positive quaternion-K\"ahler manifolds and spin symmetric spaces that have \textit{non-trivial} Rarita-Schwinger fields. In 2021, a paper of B\"{a}r and Mazzeo showed the existence of a sequence of closed simply connected negative K\"ahler-Einstein spin manifolds for which the dimension of $ker\, Q$ tends to infinity \cite{BM21}.

Often in both Yang-Mills gauge theory and classical Seiberg-Witten theory, experts use analytical methods to construct non-trivial solutions. For example in Yang-Mills gauge theory, Taubes in \cite{T82, T84} used the compactification of the ASD moduli space to show the existence of solution of the ASD equation by "pushing in" the solutions from the compactification. A similar approach was also considered in the context of other gauge-theoretic equations such as the Hitchin equation \cite{M12} and the multiple-spinor Seiberg-Witten equations \cite{P22, P23}. Our approach showing the existence of non-trivial solutions of the RS-SW equations is different and non-constructive. Roughly, in an appropriate set-up, via a hybrid method that combines functional analysis and algebraic topology, we show that the moduli space of solutions is non-compact. Equivalently, there is a sequence of solutions 
$$(A_n, \psi_n) \text{ such that } \norm{\psi_n}_{L^2} \to \infty.$$
This implies that the moduli space must contain non-trivial solutions. More comments about our method will be provided later (cf. Subsection 1.2).  We partially answer Question 1.1 by  providing a sufficient condition that ensures that there is always a non-trivial solution to (1.1) when $L$ is a trivial line bundle. In this situation, take $\nabla_{A_0}$ to be trivial and $a$ to be any purely imaginary $1-$form, then (1.1) can be rewritten as
\begin{equation}
Q^+\psi + \pi^-(a\cdot \psi) = 0, \, \, \, \, d^+ a = \rho^{-1}(\mu(\psi))
\end{equation}
Let $\mathcal{M}_g$ be the moduli space of solutions to (1.2) after the gauge fixing condition by the gauge group $\mathcal{G}:=\,Maps\,(X,U(1))$. The main result of this paper shows that

\begin{Th}
Suppose that $X$ is a closed simply connected smooth spin $4-$manifold whose intersection form is indefinite and $g$ is any Riemannian metric on $X$. If $\mathcal{M}_g$ is compact, then we must have
$$b_2(X) \geq \frac{15}{4}\,\sigma(X) + 2,$$
where $b_2(X)$ is the rank of $H^2(X)$ and $\sigma(X)$ is the signature of $X$.
\end{Th}

\begin{Cor}
Let $X$ be a closed simply connected smooth spin $4-$manifold such that $b_2(X) < 15\, \sigma(X) /4 + 2$. Then for any Riemannian metric $g$ on $X$, the moduli space mod gauge $\mathcal{M}_g$ of $(1.2)$ is always non-compact. As a result, $\mathcal{M}_g$ is non-empty and there are always non-trivial solutions to $(1.2)$.
\end{Cor}

\begin{Rem}
On $\mathbb{R}^4$, (1.1) is just a system of Seiberg-Witten equations for multiple-spinors $\{\psi_i\}_{i=0}^{3}$ that Taubes and several others have considered \cite{T16, HW15, WZ21}, but with an extra condition that $-\psi_0 + I\psi_1 + J\psi_2+K\psi_3 = 0$, where $I, J, K$ are the usual Pauli spin matrices. We hope to explore this direction in our upcomming work.
\end{Rem}

\begin{Rem}
In general, by the above theorem, $X \# S^2 \times S^2$ might still have non-compact moduli space of non-trivial solutions under a certain topological data assumed about $X$. If there is a smooth invariant associated to the RS-SW equations, then it is possible that such an invariant of $X\# S^2 \times S^2$ is non-zero. In contrast to Seiberg-Witten theory, Taubes showed that the Seiberg-Witten invariant of such a manifold would always be zero \cite{HT99}.
\end{Rem}

\begin{Rem}
Theorem 1.2 still holds for a perturbed version of (1.2). In fact, it can be shown that with appropriate perturbation and after gauge fixing, the moduli space of solutions is always a finite dimensional manifold that has an orientation depending on the choice of orientation of $H^{+}(X)$. The transversality statement is usually one of the important steps if we want to define a possible gauge theoretic invariant associated to the RS-SW equations. We will save the treatment of these results for our later work.
\end{Rem}

\begin{Rem}
In light of Corollary 1.3, it is interesting to know what should be a compactification of $\mathcal{M}_g$. We do not attempt to address this question in the present paper. However, we speculate that in some ideal geometric situation of $X$ (e.g, the metric $g$ on $X$ is Einstein with non-negative scalar curvature), any solution of $(1.2)$ of the form $(0, \psi)$  also solves a certain multiple-spinor Seiberg-Witten equations. Thus, possibly, any sequence of solutions of the type $(0, \psi_n)$ converges to a certain  so-called Fueter section after passing through a subsequence modulo gauge away from a certain singular set in $X$. At the moment, we are not aware of such a type of convergence statement  for any arbitrary sequence of solutions $(a_n, \psi_n)$ of $(1.2)$ in a less restrictive geometric setting of $X$. We reserve this question for future exploration.
\end{Rem}

\subsection{Motivation of the approach} As pointed out in the previous subsection, the existence of non-trivial solution can be potentially inferred from studying the convergence type of a sequence of solutions to the RS-SW equations. In this subsection, we briefly motivate the method we use to obtain a proof of Theorem 1.2 (Corollary 1.3). We opt for a more abstract method of finite dimensional approximation in functional analysis (cf. \cite{F01}) to study the moduli space $\mathcal{M}_g$ of the equations $(1.2)$. The key insight of our method is realizing that the defining functional $\mathcal{F}$ of $(1.2)$ is roughly a Fredholm operator modulo gauge, where its spectrum is discrete. Furthermore, $\mathcal{F}$ is not only $U(1)-$equivariant but also $Pin(2)-$equivariant. A finite approximation of $\mathcal{F}$ is roughly a map restricted to a finite direct sum of eigenspaces of $\mathcal{F}$. If $\mathcal{M}_g$ is compact (possibly empty), one can show that there exists a finite approximation of $\mathcal{F}$ that induces a $Pin(2)-$equivariant map between $Pin(2)-$representation spheres. Via a calculation in equivariant $K-$theory for $Pin(2)$, such a map cannot exist unless one has a certain bound between $b_2(X)$ and $\sigma(X)$, which is purely a topological condition. One of the advantages of our approach is that we can bypass various delicate assumptions about the geometry of $X$ (which we will describe right below) that could potentially affect $\mathcal{M}_g$.  Rather,  we focus only on the homotopy type of $\mathcal{F}$ itself to indirectly derive some information about $\mathcal{M}_g$.

We elaborate some of the difficulties if one hypothetically insists on using a more traditional analytical approach in gauge theory to prove Theorem 1.2. Following some of the conventions that were set-up in the previous subsection, the orthogonal projection $\pi$ also gives rise to another differential operator $P_A := \pi \circ \nabla_A$, which is called the twistor operator (cf. Subsection 2.1). It turns out that the kernel of $Q_A^{\pm}$ is tied intimately with $P_A^{\pm}$ in the following sense. Firstly, it can be shown that 
$$Q_A^{-}Q_A^{+} + P^+_A P^{+*}_A = \nabla^*_A \nabla_A + \frac{\kappa}{4} + \pi (\rho(F^+_A)) - \pi(1 \otimes Ric)^+,$$
where $\kappa$ and $Ric$, respectively, are the scalar curvature and the Ricci curvature of the metric $g$ on $X$. Already,  the appearance of $P_A^+P^{+*}_A$ and $Ric$ in this Weitzenb\"ock-type formula for $Q_A$ is what would set the analysis of the RS-SW equations apart from the usual multiple-spinor Seiberg-Witten equations.  Thus, the method of analyzing solutions via frequency functions by Taubes and several other people \cite{T12, T13, T14, T16, HW15, WZ21} may not be applied directly as one might have hoped in this set-up. 

Secondly, one way to deal with $P_AP^*_A$ that appears in the Weitzenb\"ock formula for $Q_A$ is by appealing to the fact that $P_A^*P_A$ is an elliptic operator. By the ellipticity of $P_A^*P_A$, we can think of $\Gamma(ker\, \rho) = ker\, P^*_A \oplus im\, P_A$ so that in the Weitzenb\"ock formula for $Q_A$ restricted to $ker \, P_A^*$, there would be no appearance of $P_AP_A^*$ anymore. One would be tempted to decouple the RS-SW equations corresponding the orthogonal decomposition of $\Gamma(ker \, \rho) = ker\, P^*_A \oplus im\, P_A$. However, such a decoupling of the RS-SW equations could only be meaningful when $Q_A$ preserves the aforementioned orthogonal decomposition of $\Gamma( ker\, \rho)$, i.e,
$$Q_A : ker\, P_A^* \to ker\, P_A^*, \quad \quad Q_A : im\, P_A \to im\, P_A.$$
It turns out that $Q_A$ behaves in such a way described above when one considers a rather restricted condition on the Riemannian metric $g$ on $X$ and the connection $U(1)-$connection $A$; e.g, $g$ is an Einstein metric and $A$ is a flat connection. Even with those assumed hypothesis, the technique of analyzing solutions in \cite{T12, T13, T14, T16, HW15, WZ21} may only be applied directly to the type of the solutions of RS-SW equations of the form $(A, \psi)$ where $A$ is flat. At the moment, we are not aware of any general analytical method that is somewhat in line with the one considered in \cite{T12, T13, T14, T16, HW15, WZ21} to analyze the solutions of the RS-SW equations when there is no additional assumption on the geometry of $X$. We note that these difficulties also arise in a simpler setting that gives the system $(1.2)$. 

\subsection{Organization}The organization of our paper is as follows:

From this point onward, $X$ is always assumed to be simply connected unless otherwise stated. In section 2, we introduce the construction of Rarita-Schwinger operators in greater details and calculate the index of a twisted Rarita-Schwinger operator. Then we re-state our variant of the Seiberg-Witten equations using the Rarita-Schwinger operator. We shall show that the equations are invariant under the gauge group action $\mathcal{G}$, and in fact there is a gauge fixing condition just as in the usual Seiberg-Witten setting. We will also prove that the equations have an extra $Pin(2)-$symmetry when $L=\underline{\CN}$. Note that this $Pin(2)-$symmetry is a special phenomena  only for "trivial" $spin^c$ structure over $X$, the technique that we use for the proof of Theorem 1.2 cannot be extended to a general $spin^c$ bundle.

In section 3, we prove various functional analysis related facts about the functional $\mathcal{F}$ of our RS-SW equations to set up for its finite dimensional approximation. Specifically, we show that the linearization $d\mathscr{Q}$ of the quadratic part $\mathscr{Q}$ of (1.2) is a $Pin(2)-$equivariant compact operator at any configuration whose norm is less than a fixed $R_0>0$. Moreover, the union over $B(0,R_0)$ of all the images of the closed unit ball via $d\mathscr{Q}$ has compact closure. We will also briefly recall the setup of finite dimensional approximation via global Kuranishi model for the sake of self-containment. Such a model was also used by Furuta in the proof of his 10/8th-Theorem \cite{F01}. In fact, we closely adapt Furuta's technique in our setting.

In section 4, we briefly recall some facts about equivariant $K-$theory and provide a proof Theorem 1.2. The use of equivariant $K-$theory in our proof is similar to Furuta's with a minor difference: we show that $\textit{if}$ $\mathcal{M}_g$ $\textit{is compact}$, then there is a finite dimensional approximation for $\mathcal{F}$ that can be viewed as $Pin(2)-$equivariant map between $Pin(2)-$representation spheres. And such a map cannot exist \textit{unless} $b_2(X) \geq 15\,\sigma(X)/4 + 2$.

In section 5, we give some concluding remarks and state potential questions for future works. 

\textit{\textbf{Acknowledgement.}} This paper is part of author's PhD dissertation thesis. The author is grateful to Jeremy Van Horn-Morris for many hours of fruitful conversation about Furuta's 10/8th-Theorem and for taking the time to clarify some of the points in the proof of the main result in this paper. We appreciate the encouragement of John Ryan who was the first person to tell the author about the Rarita-Schwinger operator. Finally, we are thankful for some of the feedbacks from Aliakbar Daemi and Greg Parker on an early manuscript of this paper.

\section{Rarita-Schwinger-Seiberg-Witten equations}
\subsection{Rarita-Schwinger operator} A Rarita-Schwinger operator can be defined on a general Riemannian manifold. Let $(X^{2n},g)$ be a closed Riemannian even dimensional manifold from now. A Clifford bundle $S\to X$ is a complex vector bundle that is equipped with a Hermitian inner product, a compatible connection $\nabla$, and a Clifford multiplication $\rho : TX \otimes S \to S$. Since the manifold is even dimensional, there is an orthogonal decomposition of $S = S^+ \oplus S^-$ such that $\rho$ exchanges the chirality.

Consider $ker\, \rho$, which is a subbundle of $TX\otimes S$ over $X$. With respect to the chirality of $S$, we also have an orthogonal decomposition $ker\, \rho = (ker\,\rho)^+ \oplus (ker\, \rho)^-$. Let $\iota : S^{\mp} \to T^*X\otimes S^{\pm} \cong TX \otimes S^{\pm}$ be the embedding defined $\iota(\varphi)(X) = -\frac{1}{2n}\rho(X)\varphi$. 

\begin{Lemma}We  have an orthogonal decomposition $S^{\pm}\otimes TX = \iota(S^{\mp})\oplus (ker\, \rho)^{\pm}$.
\end{Lemma}

\begin{proof}
Let $\{e_1,\cdots,e_{2n}\}$ be a local orthonormal frame on $X$. For any $\varphi \in \Gamma(S^{\pm})$, locally we can write
$$\iota(\varphi) = -\frac{1}{2n}\sum_{i=1}^{2n}e_i \otimes \rho(e_i)\varphi.$$
Let $\psi \in \Gamma(ker\, \rho)$. One can also write $\psi = \sum e_i \otimes \psi_i$, where $\psi_i \in \Gamma(S^+)$. Then when taking inner product of $\iota(\varphi)$ and $\psi$, we have
\begin{align}
    \Big\la -\frac{1}{2n}\sum_i e_i \otimes \rho(e_i) \varphi, \sum_j e_j \otimes \psi_j\Big\ra &= -\frac{1}{2n} \sum_{i,j}\la e_i, e_j\ra \cdot \la \rho(e_i)\varphi, \psi_j\ra \nonumber \\
    & = -\frac{1}{2n}\sum_i \la \rho(e_i \varphi, \psi_i\ra \nonumber \\
    & = \frac{1}{2n}\Big\la \varphi , \sum_i \rho(e_i)\psi_i \Big\ra = 0 \nonumber
\end{align}
As a result, we have an orthogonal decomposition of $S^{\pm}\otimes TX$ as desired.
\end{proof}

On the bundle $S$, there are two natural first order differential operators: The Dirac operator $D: \Gamma(S) \to \Gamma(S)$ defined by $D =\rho \circ \nabla$, and the twistor operator $P:\Gamma(S)\to \Gamma(ker\,\rho)$ defined by $\pi \circ \nabla$ where $\pi : TX \otimes S \to ker\, \rho$ is the orthogonal projection. Taking into account of the chirality, we note that $D^{\pm} : \Gamma(S^{\pm}) \to \Gamma(S^{\mp})$, whereas $P^{\pm} : \Gamma(S^{\pm}) \to \Gamma(ker\, \rho)^{\pm}$.

Next we consider the twisted Dirac operator $\mathcal{D}^{\pm} : \Gamma(S^{\pm} \otimes TX) \to \Gamma(S^{\mp}\otimes TX)$. Follow the computations in Wang \cite{W91}, with respect to the decomposition $S^{\pm}\otimes TX = \iota(S^{\mp})\oplus (ker\, \rho)^{\pm}$, $\mathcal{D}^{\pm}$ takes the following matrix form
$$\mathcal{D}^{\pm} = \begin{pmatrix}  \frac{1-n}{n} \iota \circ D^{\pm} \circ \iota^{-1}  & 2\iota \circ P^{\pm*}\\ \frac{1}{n}P^{\mp} \circ \iota^{-1} & Q^{\pm} \end{pmatrix}.$$

\begin{Def}
The Rarita-Schwinger operator associated to the Clifford bundle $S\to X$ is defined by $Q^{\pm} :=\pi\circ \mathcal{D}^{\pm}|_{\Gamma(ker\,\rho)^{\pm}}$. 
\end{Def}

The above construction carries out similarly when $(X,g)$ is a closed smooth spin $4-$manifold. In this case, a Clifford bundle is given by a choice of $spin^c$ bundle $\mathfrak{s}_{1/2}\otimes L \to X$ where $\mathfrak{s}_{1/2}$ is the associated vector bundle to a principal $Spin(4)-$bundle $P_{Spin(4)} \to X$ and $L$ is any complex line bundle. Notation-wise, the operators defined previously all come with extra subscript: $D^{\pm}_A, \mathcal{D}^{\pm}_A, Q^{\pm}_A$ with $A$ being a unitary connection on $L$. Borrowing the language from Physics, we call sections of $\mathfrak{s}^{\pm}_{1/2}\otimes L$ positive (negative) twisted $1/2-$spinors and sections of $(ker\, \rho)^{\pm}:=\mathfrak{s}^{\pm}_{3/2}\otimes L$ positive (negative) twisted $3/2-$spinors.

Rarita-Schwinger operator $Q_A$ is a first order elliptic operator. Thus it has a well-defined index.

\begin{Prop}
Let $X$ be any closed  smooth spin $4-$manifold ($X$ does not have to be simply connected) and $Q_A$ be the Rarita-Schwinger operator associated to the twisted $3/2-$spinor bundle $\mathfrak{s}_{3/2} \otimes L \to X$ where $A$ is a unitary connection of $L$. Then the index of $Q_A$ is given by
$$index_{\CN}\,Q_A = \frac{19}{8}\,\sigma(X) + \frac{5}{2}\,c_1(L)^2.$$
\end{Prop}

\begin{proof}
By Theorem 13.13 in \cite{LM89} or Proposition 2.17 in \cite{AS68}, which is the equivariant version of the Atiyah-Singer index theorem for first order elliptic operator, we have
\begin{equation}
index_{\CN}\, Q_A = \left(\frac{ch(\mathfrak{s}^+_{3/2}\otimes L) - ch(\mathfrak{s}^-_{3/2}\otimes L)}{\chi(TX)}\,\hat{\mathcal{A}}(TX)^2\right)[X],
\end{equation}
where $\chi(TX)$ and $\hat{\mathcal{A}}(TX)$ are the usual Euler class and the $\hat{A}-$class of $TX$, $ch$ is the Chern character. On the other hand, using the orthogonal decomposition of $\mathfrak{s}^{\pm}_{1/2}\otimes L\otimes TX$, we obtain
\begin{equation}
ch(\mathfrak{s}^{\pm}_{1/2}\otimes L)\,ch(TX) = ch(\mathfrak{s}^{\pm}\otimes L \otimes TX) = ch(\mathfrak{s}^{\pm}_{3/2}\otimes L) + ch(\mathfrak{s}^{\mp}_{1/2}\otimes L).
\end{equation}
Subtract the two version of (2.2) from each other to get
\begin{equation}
ch(\mathfrak{s}^+_{3/2}\otimes L) - ch(\mathfrak{s}^-_{3/2}\otimes L) = (ch(\mathfrak{s}^{+}_{1/2}\otimes L)-ch(\mathfrak{s}^{-}_{1/2}\otimes L))(ch(TX) + 1).
\end{equation}
Substitute (2.3) into (2.1), recall that $\hat{\mathcal{A}}(TX) = (ch(\mathfrak{s}^+_{1/2}-ch(\mathfrak{s}^{-}_{1/2}))\,\hat{\mathcal{A}}(TX)^2/\chi(TX)$ to get
\begin{align}
index_{\CN}\,Q_A &= ch(TX\otimes L)\hat{\mathcal{A}}(TX)[X] + ch(L)\hat{\mathcal{A}}(TX)[X] \nonumber \\
&= index_{\CN}\,\mathcal{D}_A +  \left( -\frac{1}{8}\sigma(X) + \frac{1}{2}c_1(L)^2\right)
\end{align}
Let's calculate $ch(TX\otimes L) = ch(T_{\CN}X\otimes L)$. Recall:

\textbf{Fact 1.} $ch(T_{\CN}X\otimes L) = ch(\mathfrak{s}^+_{1/2} \otimes \mathfrak{s}^-_{1/2} \otimes L) = ch(\mathfrak{s}^+_{1/2})ch(\mathfrak{s}^-_{1/2})ch(L)$.

Note that $ch(\mathfrak{s}^+_{1/2})ch(\mathfrak{s}^-_{1/2}) = 4 - c_2(\mathfrak{s}^+_{1/2}) - 2c_2(\mathfrak{s}^-_{1/2})$, and $ch(L) = 1 + c_1(L) + c_1(L)^2 /2$. Thus, $ch(T_{\CN}\otimes L) = 4 + 2c_1(L)^2 - 2c_2(\mathfrak{s}^+_{1/2}) - 2 c_2(\mathfrak{s}^-_{1/2}) + 4c_1(L)$. Combine with the fact that on $4-$manifolds, $\hat{\mathcal{A}}(TX) = 1 - p_1(TX)/24$, where $p_1$ is the first Pontryagin class, we have (picking out top forms)
\begin{equation}
\hat{\mathcal{A}}(TX)ch(T_{\CN}X \otimes L) = 2c_1(L)^2 - 2c_2(\mathfrak{s}^+_{1/2}) - 2 c_2(\mathfrak{s}^-_{1/2}) - \frac{p_1(TX)}{6}.
\end{equation}

Recall:

\textbf{Fact 2.} $c_2(\mathfrak{s}^{\pm}_{1/2})[X] = -3\,\sigma(X)/4 \mp \chi(X)/2$.

Since $p_1(TX)/3[X] = \sigma(X)$, use the above formulas, we rewrite $(2.5)$ as
\begin{align}
&\hat{\mathcal{A}}(TX)ch(T_{\CN}X \otimes L)[X]= 2c_1(L)^2[X] - 2\left(-\frac{3}{4}\,\sigma(X) - \frac{\chi(X)}{2}\right)+\\
&-2\left(-\frac{3}{4}\sigma(X) + \frac{\chi(X)}{2}\right) - \frac{\sigma(X)}{2} = 2c_1(L)^2[X] + \frac{5}{2}\sigma(X).
\end{align} 
Therefore, $(2.4)$ becomes $index_{\CN}\,Q_{A} = 19\,\sigma(X)/8 + 5c_1(L)^2/2$.
\end{proof}

\begin{Rem}
When $L = \underline{\CN}$ and $A$ is the trivial connection, from Proposition 2.2 we have $index_{\CN}\,Q = 19\,\sigma(X)/8$. 
\end{Rem}

\subsection{Rarita-Schwinger-Seiberg-Witten equations} By Lemma 4.55 in \cite{S14}, the Clifford multiplication $\rho$ is an isometry and an isomorphism from $i\Lambda^+T^*X \to i\mathfrak{su}(\mathfrak{s}^+_{1/2})$. Let $\mu$ be a quadratic map defined by $\mu: \mathfrak{s}^{+}_{1/2} \otimes L \otimes TX = \mathfrak{s}^{+}_{1/2} \otimes L\otimes T^*X = \text{Hom}(TX, \mathfrak{s}^+_{1/2} \otimes L) \to \mathfrak{g}_L \otimes \mathfrak{su}(\mathfrak{s}^+_{1/2}) = i\mathfrak{su}(\mathfrak{s}^+_{1/2}),$
$$\mu(\psi) = \psi \psi^* - \frac{1}{2}\text{tr}(\psi \psi^*)1_{\mathfrak{s}^{+}_{1/2}\otimes L}.$$

Besides looking for twisted Rarita-Schwinger fields, we impose a curvature condition for $A$. What follow are referred as the Rarita-Schwinger-Seiberg-Witten equations (RS-SW),
\begin{equation}
Q^{+}_{A} \psi = 0, \,\,\, F^+_{A} = \rho^{-1}(\mu(\psi)).
\end{equation}

If $\nabla_{A_0}$ is a fixed referenced unitary connection of $L$, then every other connection $\nabla_A = \nabla_{A_0} + a$, where $a \in i\Omega^1(X)$. (2.8) can also be re-written as
\begin{equation}
Q^+_{A_0}\psi + \pi^-(a\cdot \psi)  = 0,\,\,\, d^{+}a + F^+_{A_0} = \rho^{-1}(\mu(\psi)).
\end{equation}

We define $\mathcal{C} = \Gamma(\mathfrak{s}^+_{3/2} \otimes L) \oplus i\Omega^1(X)$ to be the configuration space of the RS-SW equations. $\mathcal{R} = \Gamma(\mathfrak{s}^-_{3/2} \otimes L) \oplus i\Omega^+(X)$ is denoted by its range space. The gauge group $\mathcal{G}=\,Maps(X,U(1))$ acts on $\mathcal{C}$ by pulling-back the connections and left multiplying by conjugation on the twisted $3/2-$spinors. $\mathcal{G}$ also acts on $i\Omega^+(X)$ trivially. Just as in the standard Seiberg-Witten theory, the following lemma tells us that solutions to (2.9) are preserved under the gauge group action.

\begin{Lemma}
If $(\psi, A_0+a)$ is a solution to (2.9), then $h \cdot (\psi, A_0+a)$ is also a solution for any $h \in \mathcal{G}=Maps\,(X,U(1))$. 
\end{Lemma}

\begin{proof}
Since $X$ is simply connected, $h$ has global logarithm, i.e, there is a smooth real-valued function $u$ such that $h = e^{iu}$. Then the action of $h$ on a configuration can be re-described as
$$h\cdot(\psi, A_0+a) = (e^{-iu}\psi, A_0 + a + idu).$$
We easily see that $d^+(a + idu) + F^+_{A_0} - \rho^{-1}(\mu(e^{-iu}\psi)) = d^+a+F^+_{A_0} - \rho^{-1}(\mu(\psi))$. At the same time, 
\begin{align}
Q^+_{A_0}(e^{-iu}\psi) &+ \pi^{-}(a\cdot e^{-iu}\psi + idu\cdot e^{-iu}\psi) = e^{-iu}Q^+_{A_0}\psi -ie^{-iu} \pi^-(du\cdot \psi) +\\
&+ e^{-iu}\pi^{-}(a\cdot \psi)  + ie^{-iu}\pi^-(du \cdot \psi) = e^{-iu}(Q^+_{A_0}\psi + \pi^{-}(a\cdot \psi)).
\end{align}
Therefore, obviously if $(\psi,A_0+a)$ solves (2.9), then so does $(e^{-iu}\psi, A_0 + a + idu)$. 
\end{proof}

RS-SW equations are in an abelian gauge theory. And just as the Seiberg-Witten equations, there is a gauge-fixing condition for the RS-SW equations. Recall that with respect to a referenced unitary connection $A_0$, $A$ is in the Coulomb gauge if $d^*(A-A_0) = 0$.

\begin{Lemma}
Let $(\psi, A_0+a)$ be any solution to (2.9). Then its orbit by the action of $\mathcal{G}$ can be represented uniquely up to a constant by another solution where the connection part is in the Coulomb gauge.
\end{Lemma}

\begin{proof}
Suppose $h=e^{iu}$ is an element of $\mathcal{G}$ such that $A_0+a+idu$ is in the Coulomb gauge. That means to show the existence of a connection that is in the Coulomb gauge, we have to solve $d^*(a+idu)=0$. Note that since $u$ is a smooth function, $d^*u = 0$. Therefore, solving for $u$ in $d^*(a+idu)=0$ is equivalent to solving for $u$ in 
\begin{equation}
d^*du+dd^*u=\Delta u = d^*(-ia)
\end{equation}
But by Hodge's decomposition theorem, $\Omega^0(X) = ker\, \Delta \oplus d^*\Omega^1(X)$, which means that a solution $u$ for (2.12) always exists up to a constant. 
\end{proof}

\begin{Rem}
Lemma 2.5 and Lemma 2.6 should also hold when $X$ is not simply connected. However, the proof of gauge slice condition is much simpler in the simply connected case due to the existence of global exponential lift of any gauge transform.
\end{Rem}

From Lemma 2.5 and Lemma 2.6, we see that after the gauge group action and the gauge fixing condition, RS-SW becomes a non-linear system of elliptic differential equations
\begin{equation}
Q^+_{A_0}\psi + \pi^-(a\cdot \psi) = 0,\, \, \, d^*a = 0, \, \, \, d^+ a + F^+_{A_0} = \rho^{-1}(\mu(\psi)).
\end{equation}
When $L=\underline{\CN}$ and $A_0$ is chosen to be the trivial connection, equations (2.13) become
\begin{equation}
Q^+\psi + \pi^{-}(a\cdot \psi) = 0,\,\,\, d^*a = 0,\,\,\, d^+a = \rho^{-1}(\mu(\psi)).
\end{equation}
The configuration space $\mathcal{C}$ and the range space $\mathcal{R}$ of (2.14) become $\Gamma(\mathfrak{s}^+_{3/2})\oplus i\Omega^1(X)$ and $\Gamma(\mathfrak{s}^{-}_{3/2})\oplus i\Omega^0(X)\oplus i\Omega^{+}(X)$. Denote $\mathcal{F} = \mathscr{D}\oplus \mathscr{Q} : \mathcal{C} \to \mathcal{R}$ by the functional of (2.14), where $\mathscr{D} = Q^+ \oplus(d^+\oplus d^*)$ and $\mathscr{Q}(\psi, a) = (\pi^-(a\cdot \psi), -\rho^{-1}(\mu(\psi)), 0)$. Note that $\mathcal{M}_g = \mathcal{F}^{-1}(0)$. For the remaining of the paper, we focus only on this functional $\mathcal{F}$.

\subsection{$Pin(2)-$ equivariance of $\mathcal{F}$}Let $V$ be a real $4-$dimensional vector space. $V$ can be given a quaternionic structure by considering each vector $v$ in $V$ as the following $2 \times 2-$complex matrix
$$v = \begin{pmatrix} a+bi & -c+di\\ c+di & a-bi \end{pmatrix} := \begin{pmatrix} z & -\overline{w} \\ w & \overline{z} \end{pmatrix}$$
This turns $V$ into an algebra where the multiplication is defined by multiplication of these complex matrices. Furthermore, the usual Euclidean inner product on $V$ agrees with the Frobenius inner product on $V$ equipped with the quaternionic structure. We note that if $v$ is not trivial, then $v = |v|^2 \cdot (v/ |v|^2)$. And since $|v|^2 = \det\, v$, we can also think of $V$ as $\RN\cdot SU(2)$. 

Let $W = W^+ \oplus W^-$, where $W^{\pm}$ each is a copy of $\mathbb{C}^2$. We define an $\mathbb{R}-$linear map $\rho: V \to \text{End}(W)$ by
$$\rho(v) = \begin{pmatrix} 0 & -\overline{v}^t \\ v & 0 \end{pmatrix}.$$
Let $\{e_1, e_2, e_3, e_4\}$ be any orthonormal basis of $V$. Then $\rho(e_i)\rho(e_j) + \rho(e_j)\rho(e_i) = -2 \delta_{ij}$. So $\rho$ is a Clifford multiplication. Furthermore, $\rho$ exchanges the chirality of $W$. In particular if $\psi \in W^{+}$, then $\rho(v)\psi = v\,\psi$; where we view $v\, \psi$ as an element of $W^-$. Also note that $W^{\pm}$ corresponds to $\pm 1-$eigenspaces of the map $-\rho(e_1)\rho(e_2)\rho(e_3)\rho(e_4)$. 

$\CN^2$ is also naturally endowed with a quaternionic structure. Let $\psi \in \CN^2$, and write $\psi = (\psi_1 \, \, \psi_2)^t$. We can view $\psi$ as a quaternion number in two different ways, either $\psi = \psi_1 + \psi_2 j$ or 
$$\psi = \begin{pmatrix} \psi_1 & -\overline{\psi_2} \\ \psi_2 & \overline{\psi_1} \end{pmatrix}.$$

Now consider the expression $\rho(v)\psi$, where $v \in V$ and $\psi \in W^+$. As observed above, we have
$$\rho(v)\psi = \begin{pmatrix} z & -\overline{w} \\  w & \overline{z} \end{pmatrix}  \begin{pmatrix} \psi_1 \\ \psi_2 \end{pmatrix} = \begin{pmatrix} z\psi_1 -\overline{w}\psi_2 \\ w\psi_1 + \overline{z} \psi_2\end{pmatrix} \in W^-.$$

The above vector then can be identified with the matrix
$$\begin{pmatrix} z\psi_1 - \overline{w}\psi_2 & -\overline{w\psi_1} - z\overline{\psi_2} \\ w\psi_1 + \overline{z}\psi_2 & \overline{z\psi_1} - w \overline{\psi_2} \end{pmatrix} =  \begin{pmatrix} z & -\overline{w} \\ w & \overline{z} \end{pmatrix}  \begin{pmatrix} \psi_1 & -\overline{\psi_2} \\ \psi_2 & \overline{\psi_1} \end{pmatrix}, $$
which is exactly the multiplication of $v$ and $\psi$ in the quaternion numbers. Hence, the representation (Clifford multiplication) is exactly the multiplication in the quaternions.

Consider the group $Spin(4) = SU(2) \times SU(2)$. Let $(p_{-}, p_{+}) \in \text{Spin}(4)$ arbitrarily. We define the following group homomorphism $Spin(4) \to GL(V)$ given by
$$ (p_{-}, p_{+}) \mapsto (v \mapsto p_{-} v p_{+}^{-1}).$$

The map $v \mapsto p_{-}vp_{+}^{-1}$ is norm preserving; thus the above homomorphism is actually $Spin(4) \to SO(V)$. Either way, this is a representation of $Spin(4)$ on $V$. There is another representation of $Spin(4)$ onto $W$ defined by 
$$(p_{-}, p_{+}) \mapsto ((\psi, \phi) \mapsto (p_{-}\phi, p_{+}\psi)).$$

This is the spinor representation of the spin group. When $Spin(4)$ only acts on $W^+$, the above action agrees with the Clifford multiplication $\rho$ restricted to only unit vectors in $V$ (as multiplication in the quaternions). 

The group $Pin(2)$ is defined to be the double cover of $O(2)$
$$Pin(2) := \Big\{ \begin{pmatrix} e^{i\theta} & 0 \\  0 & e^{-i\theta} \end{pmatrix}, \begin{pmatrix} 0 & e^{i\theta} \\  e^{-i\theta} & 0 \end{pmatrix} : \theta \in \RN\Big\}.$$
Topologically, $Pin(2)$ is the disjoint union of two unit circles, $Pin(2) = S^1 \sqcup j\cdot S^1$. But we can also view $S^1 \sqcup j \cdot S^1 \subset SU(2)$. Let $p_0 \in Pin(2)$ and $\psi, \phi \in W^{\pm}$, respectively. There is another action on $W^{\pm}$ by $Pin(2)$ given by $\psi \mapsto \psi p_0^{-1}$ and $\phi \mapsto \phi p_0^{-1}$. We would like to emphasize again that every action we have listed so far is defined as multiplication in the quaternions. Following Furuta's notations \cite{F01}, we obtain the following $Spin(4)\times Pin(2)-$modules:
\begin{itemize}
	\item $V:={}_{-}\mathbb{H}_{+}$ is the module where the action of $(p_{-}, p_{+}, p_0)$ is given by $p_{-}vp_{+}^{-1}$. 
	\item $W^+ := {}_{+}\mathbb{H}$ is the module where the action is given by $p_{+}\psi p_0^{-1}$.
	\item $W^{-} := {}_{-}\mathbb{H}$ is the module where action is defined by $p_{-}\phi p_0^{-1}$.
	\item $\tilde{\mathbb{R}}$ is the $Pin(2)-$module where the actions are $e^{i\theta}\cdot x = x$ and $j \cdot x = -x$. ${}_{+}\mathbb{H}_{+}$ is the module where the action is given by $p_{+} v p_{+}^{-1}$ so that $\tilde{\mathbb{R}}\oplus \Lambda^{+}V$ is identified with ${}_{+}\mathbb{H}_{+}$. 
\end{itemize}

On our manifold $X$, with respect to its only $Spin(4)-$bundle $P_{Spin(4)}\to X$, the above actions make the bundles $T^*X = TX, \mathfrak{s}^{\pm}_{1/2}$ equivariant with respect to $Spin(4)\times Pin(2)$. Since $Pin(2) \xhookrightarrow{} Spin(4)$ by the diagonal map, these bundles are also $Pin(2)-$equivariant. Fiber-wise, they correspond to ${}_{-}\mathbb{H}_{+}$, ${}_{+}\mathbb{H}$, and ${}_{-}\mathbb{H}$, respectively, and the fiber of $\underline{\mathbb{R}}\oplus \Lambda^{+}T^{*}X$ corresponds to ${}_{+}\mathbb{H}_{+}$. We note that the Clifford multiplication $\rho : TX \otimes \mathfrak{s}^{\pm}_{1/2} \to \mathfrak{s}^{\mp}_{1/2}$ is induced exactly by the Clifford multiplication ${}_{-}\mathbb{H}_{+} \times {}_{\pm}\mathbb{H} \to {}_{\mp}\mathbb{H}$ defined previously; and the later one is definitely $Pin(2)-$equivariant. Thus, $\rho :TX \otimes \mathfrak{s}^+_{1/2}\to \mathfrak{s}^{-}_{1/2}$ is $Pin(2)-$equivariant.

\begin{Lemma}
The bundles $\mathfrak{s}^{\pm}_{3/2}$ are $Spin(4)\times Pin(2)-$equivariant.
\end{Lemma}

\begin{proof}
We just need to check this for positive $3/2-$spinors. Let $\psi = \sum e_{\alpha} \otimes \psi_{\alpha} \in \Gamma(\mathfrak{s}^{+}_{1/2})$, where $\{e_\alpha\}$ is a local orthonormal basis on $X$ with respect to $g$. Introduce the action of $(p_{-}, p_{+}, p_0)$ to $\psi$, we obtain
$\rho((p_{-},p_{+},p_0)\cdot \psi )= \rho\left(\sum p_{-}e_ap^{-1}_{+} \otimes p_{+}\psi_a p^{-1}_0\right) = p_{-}\rho (\psi) p^{-1}_0 = 0$.
Here we use the equivariant property of $\rho$. Therefore, if $\psi$ is a $3/2-$spinor, then so is $(p_{-},p_{+},p_0)\cdot \psi$. Note that this means that the $3/2-$spinor bundles are also $Pin(2)-$equivariance.
\end{proof} 

With abuse of notation, recall that we also denote the Clifford multiplication $TX \otimes \mathfrak{s}^{+}_{1/2} \otimes TX \to \mathfrak{s}^{-}_{1/2}\otimes TX$ by $\rho$. 

\begin{Lemma}
For $a \in \Omega^1(X)$ and $\psi \in \Gamma(\mathfrak{s}^{+}_{3/2})$, $(a, \psi) \mapsto a\cdot\psi$ is $Spin(4)\times Pin(2)-$equivariant. 
\end{Lemma}

\begin{proof}
Write $\psi = \sum e_{\alpha} \otimes \psi_{\alpha}$. Then $\rho(a)\psi = \sum e_{\alpha} \otimes \rho(a)\psi_{\alpha}$. First, we have
$(p_-,p_+,p_0)\cdot \rho(a)\psi = \sum p_{-}e_{\alpha}p_{+}^{-1} \otimes p_{-} a\psi_{\alpha}p_0^{-1}$.
On the other hand, 
$$\rho(p_{-}ap_{+}^{-1})\sum p_{-}e_{\alpha}p^{-1}_{+} \otimes p_{+}\psi_{\alpha}p_0^{-1} = \sum p_{-}e_{\alpha}p^{-1}_{+} \otimes p_{-}a\psi_{\alpha}p^{-1}_0.$$
The two expressions are exactly the same; hence we have the equivariant property of the quadratic map. Again, this also implies $Pin(2)-$equivariant.
\end{proof}

Recall that the orthogonal projection $\pi^{-} : \Gamma(\mathfrak{s}^{-}_{1/2} \otimes TX) \to \Gamma(\mathfrak{s}^{-}_{3/2})$ can be explicitly given by $\pi^{-}(\phi) = \phi - \iota \circ \rho(\phi) = \phi + \frac{1}{4} \sum_{\alpha} e_{\alpha} \otimes e_{\alpha}\rho(\phi)$. 

\begin{Lemma}
$\pi^{-}$ is $Pin(2)-$equivariant.
\end{Lemma}

\begin{proof}
Let $\phi = \sum e_{\alpha} \otimes \phi_{\alpha}$. We only need to check for the action of $j$. First, we note
\begin{align}
\pi^{-} (j\cdot \phi) & = j\cdot \phi - \iota \circ \rho(j\cdot \phi) = j\cdot \phi - \iota \circ \rho \left( \sum j e_{\alpha} j^{-1} \otimes j \phi_{\alpha} j^{-1} \right)\nonumber\\
& = j\cdot \phi - \iota ( j\rho(\phi)j^{-1})= j \cdot \phi + \frac{1}{4} \sum_{\beta} e_{\beta} \otimes e_{\beta}j\rho(\phi)j^{-1}\ \nonumber\\
& = j\cdot \phi + \frac{1}{4} \sum_{e_{\beta} \neq j} e_{\beta} \otimes e_{\beta}j\rho(\phi)j^{-1} -\frac{1}{4} j \otimes \rho(\phi)j^{-1}.\nonumber
\end{align}
On the other hand, we have
\begin{align}
j \cdot \pi^{-}(\phi) & = j \cdot (\phi - \iota \circ \rho(\phi)) = j \cdot \phi + \frac{j}{4} \sum_{\beta} e_{\beta} \otimes e_{\beta}\rho(\phi)\nonumber\\
&= j \cdot \phi + \frac{1}{4} \sum_{\beta} j e_{\beta} j^{-1} \otimes j e_{\beta} \rho(\phi) j^{-1} \nonumber\\
& = j\cdot \phi +\frac{1}{4} \sum_{e_{\beta} \neq j} e_{\beta} \otimes e_{\beta} j \rho(\phi)j^{-1} - \frac{1}{4} j \otimes \rho(\phi)j^{-1}.\nonumber
\end{align}
As a result, we have $\pi^{-} (j\cdot \phi) = j \cdot \pi^{-}(\phi)$.
\end{proof}

Let $\psi \in \mathfrak{s}^+_{3/2}$. Consider the action of $(p_-,p_+,p_0)$ on $\psi$ to obtain $p_+ \psi p_0^{-1}$. Passing through quadratic map $\mu$, we have
$$p_+ \psi p^{-1}_0 p_0 \psi^* p^{-1}_+ - \frac{1}{2}\text{tr}(p_+ \psi p^{-1}_0 p_0 \psi^* p^{-1}_+ )1 = p_+ \psi \psi^* p^{-1}_+ - \frac{1}{2}\text{tr}(\psi \psi^*) p_+ p^{-1}_+ = p_+ \mu(\psi)p^{-1}_+ $$

This shows that

\begin{Lemma}
$\mu : \mathfrak{s}^+_{3/2} \to i\mathfrak{su}(\mathfrak{s}^+_{1/2})$ is equivariant with respect to $\text{Spin}(4)\times Pin(2)$; hence $Pin(2)-$equivariant.
\end{Lemma}

All aforementioned $Pin(2)$ actions have been unitary (or orthogonal). Therefore, the connections defined on the bundles are also $Pin(2)-$equivariant. In particular, $\mathcal{D}^{+}$ and $d + d^*$ are $Pin(2)-$equivariant. Combine with the previous lemmas, we obtain 

\begin{Cor}The functional $\mathcal{F}: \mathcal{C} \to \mathcal{R}$ is a $Pin(2)-$equivariant map. 
\end{Cor}

\section{Finite dimensional approximation of $\mathcal{F}$}

\subsection{Linearization of $\mathcal{F}$}We consider the Sobolev completions of $\mathcal{C}_{L^2_{k}}$ and $\mathcal{R}_{L^2_{k-1}}$ by $L^2_{k}$ and $L^2_{k-1}$ norms, respectively. Naturally upon completion, they become Hilbert manifolds. As such, for the sake of convenience, we suppress their subscripts. There is a natural identification of the tangent spaces of the Hilbert manifolds $\mathcal{C}$ and $\mathcal{R}$ at a point $(\psi, a)$. Consider a parametrization of a curve $\gamma_t := (\psi_t, a_t)$ in $\mathcal{C}$ for $t \in (-\epsilon, \epsilon)$, where $\epsilon > 0$ and 
$\psi_t = \psi + t \phi, \,\,\, a_t =  a + tb$. Here $\phi$ is some other positive $3/2-$spinor  and $b$ is a purely imaginary valued $1-$form. Then a short calculation shows that the derivative of $\gamma_t$ at $t = 0$ gives us $\dot \gamma_0 = (\phi, b)$.

\begin{Lemma}
$\mathcal{F} : \mathcal{C} \to \mathcal{R}$ is a smooth mapping between separable Hilbert manifolds and its differential at a point $(\psi,a)$ is
$$d_{(\psi,a)} \mathcal{F}(\phi, b) = \left( Q^+ \phi + \pi^{-}(b\cdot \psi + a\cdot \phi),\,\, d^*b,\,\, d^+b - \rho^{-1}(\psi \phi^* + \phi \psi^*)_0\right)$$
\end{Lemma}

\begin{proof}
Smoothness part comes from Sobolev regularity. Now let $\gamma_t = (\psi_t, a_t)$, $t\in (-\epsilon, \epsilon)$ such that $\gamma_0 = (\psi, a)$ and $\dot \gamma_0 = (\phi, b)$, where $\phi \in \Gamma(\mathfrak{s}^+_{3/2})$ and $b \in i\Omega^1(X)$. We have
\begin{align}
&d_{(\psi,a)} \mathcal{F} (\phi, b) = \lim_{t\to 0} \frac{\mathcal{F}(\gamma_t) - \mathcal{F}(\psi,a)}{t} \nonumber \\
&= \left( Q^+\phi + \lim_{t\to 0} \frac{\pi^{-} (a_t\cdot\psi_t - a\cdot\psi)}{t},\,\,  d^*b,\,\, d^+b-\lim_{t\to 0} \frac{\rho^{-1}(\mu(\psi_t) - \mu(\psi))}{t} \right).
\end{align}
\noindent
Compute the first limit on the right hand side of (3.1), we have
\begin{align}
\lim_{t\to 0} \frac{\pi^{-}(a_t\cdot\psi_t - a\cdot\psi)}{t}&= \pi^{-}\left(\lim_{t\to 0} \frac{a_t\cdot\psi_t - a_t\cdot\psi}{t} + \lim_{t\to 0} \frac{a_t - a}{t}\cdot\psi\right) \nonumber \\
& =  \pi^{-}( b\cdot\phi + a\cdot\psi).
\end{align}
\noindent
For the other limit, we see that
\begin{align}
&\frac{(\psi_t\psi^*_t)_0 -(\psi \psi^*)_0}{t} = \frac{\psi_t \psi^*_t - \psi \psi^*}{t} -\frac{1}{2} \text{tr}\left(\frac{\psi_t \psi^*_t - \psi \psi^*}{t}\right)1 \nonumber \\
& = \frac{\psi_t(\psi^*_t - \psi^*)}{t} + \frac{(\psi_t - \psi) \psi^*}{t} -\frac{1}{2} \text{tr} \Big\{ \frac{\psi_t(\psi^*_t -\psi^*)}{t} + \frac{(\psi_t -\psi)\psi^*}{t}\Big \}1.
\end{align}
Let $t \to 0$ and (3.3) approaches $(\psi \phi^* + \phi \psi^*) - \frac{1}{2} \text{tr}(\psi \phi^* + \phi \psi^*) \text{1}$, which is exactly $(\psi \phi^* + \phi \psi^*)_0$.
The above calculations give us exactly the formula for the differential of $\mathcal{F}$ at $(\psi,a) \in \mathcal{C}$. 
\end{proof}

\begin{Lemma}
Suppose $k \geq 4, R_0 >0$ and $\norm{(\psi,a)}_{L^2_k} < R_0$. The linear operator $R_{(\psi,a)} : L^2_{k}(\mathfrak{s}^{+}_{3/2}) \oplus L^2_{k}(iT^*X) \to L^2_{k-1}(\mathfrak{s}^-_{1/2}\otimes TX)$ given by $R_{(\psi,a)}(\phi,b) = b\cdot \psi + a \cdot \phi$ is a compact operator.
\end{Lemma}

\begin{proof}
Let $\{(\phi_j,b_j)\}_{j=1}^{\infty}$ be a uniformly bounded sequence in $ L^2_{k}(\mathfrak{s}^{+}_{3/2}) \oplus L^2_{k}(iT^*X)$. We need to show that there is a subsequence of $\{R_{\psi,a}(\phi_j,b_j)\}_{j=1}^{\infty}$ that converges in $L^2_{k-1}(\mathfrak{s}^-_{1/2}\otimes TX)$. By the Rellich lemma, there is a subsequence of $\{b_j\}_{j=1}^{\infty}$ that converges in $L^2_{k-1}$. Without loss of generality, we may denote such subsequence by $\{b_j\}_{j=1}^{\infty}$ again and suppose that $b_j \to b$ in $L^2_{k-1}(iT^*X)$. For an arbitrary $\epsilon > 0$, let $j$ be large enough such that $\norm{b_j - b}_{L^2_{k-1}} < \epsilon/2\text{const}R_0$. Then by the Sobolev multiplication theorem, we have $\norm{(b_j -b)\cdot\psi}_{L^2_{k-1}} \leq \text{const} \norm{b_j - b}_{L^2_{k-1}}R_0 < \epsilon/2$. Similarly, for $j$ also large enough, we also have $\norm{a\cdot(\phi_j - \phi)}_{L^2_{k-1}} < \epsilon/2$. Therefore, $\{R_{(\psi,a)}(\phi_j,b_j)\}_{j=1}^{\infty}$ also converges in $L^2_{k-1}(\mathfrak{s}^-_{1/2}\otimes TX)$. It is not hard to see that $R_{(\psi,a)}$ is bounded. This shows that $R_{(\psi,a)}$ is indeed a compact operator.
\end{proof}

\begin{Lemma}
When $k \geq 4$, $R_0 >0$, the linear map $L_{\psi}: L^2_k(\mathfrak{s}^+_{3/2}) \to L^2_{k-1}(i\Lambda^+ T^*X)$ defined by $L_{\psi}(\phi) =  \rho^{-1}(\psi\phi^* + \phi \psi^*)_0$ is a compact operator for any $\psi$ such that $\norm{\psi}_{L^2_k} < R_0$.
\end{Lemma}

\begin{proof}
It is sufficient to show that $\phi \mapsto \psi\phi^* + \phi \psi^*$ is compact. Suppose $\{\phi_j\}_{j=1}^{\infty}$ is a uniformly bounded sequence in $L^2_k(\mathfrak{s}^+_{3/2})$. The Rellich lemma tells us that there is a subsequence of $\{\phi_j\}_{j=1}^{\infty}$ that converges in $L^2_{k-1}(\mathfrak{s}^+_{3/2})$. For notational convenience, we denote such subsequence by $\{\phi_j\}_{j=1}^{\infty}$ again, and $\phi_j \to \phi$ in $L^2_{k-1}(\mathfrak{s}^+_{3/2})$. Now for any $\epsilon > 0$ and $j$ is large enough, we have $\norm{\phi_j - \phi}_{L^2_{k-1}} < \epsilon/ 2\text{const} R_0$. By the Sobolev multiplication theorem, 
\begin{align}
&\norm{\psi(\phi^*_j - \phi^*_j) + (\phi_j - \phi)\psi^* }_{L^2_{k-1}} \leq \norm{\psi(\phi^*_j - \phi^*_j)}_{L^2_{k-1}} + \norm{(\phi_j - \phi)\psi^*}_{L^2_{k-1}} \leq \nonumber \\
& \leq \text{const} R_0 \norm{\phi^*_j - \phi^*}_{L^2_{k-1}} + \text{const} \norm{\phi_j - \phi}_{L^2_{k-1}} R_0.
\end{align}
The last inequality of (3.4) is less than $\epsilon$. As a result, $L_{\psi}$ is compact as desired.
\end{proof}

Note that the formula in Lemma 3.1 is exactly $d_{(\psi,a)}\mathcal{F} = \mathscr{D} + d_{(\psi, a)} \mathscr{Q}$. Since orthogonal projection $\pi^-$ is bounded and Lemma 3.2 tells us that $R_{(\psi,a)}$ is compact whenever $\norm{(\psi,a)}_{L^2_k}<R_0$, $\pi^-\circ R_{(\psi,a)}$ is also compact. By Lemma 3.3, we immediately see that $d_{(\psi,a)}\mathscr{Q}$ is a compact operator whenever $\norm{(\psi,a)}_{L^2_k}<R_0$. Obviously, the linearization of $\mathcal{F}$ and $\mathscr{Q}$ are also $Pin(2)-$equivariant. To summarize the point of this subsection, we have the following proposition.

\begin{Prop}
The operator $d_{(\psi, a)} \mathscr{Q} = d_{(\psi,a)}\mathcal{F} - \mathscr{D} : \mathcal{C} \to \mathcal{R}$ is compact and equivariant under the $Pin(2)-$action. Furthermore, the set
$$\bigcup_{(\psi,a) \in B(0,R_0)} d_{(\psi,a)}\mathscr{Q}(\overline{B(0,1)})$$
has compact closure in $\mathcal{R}$ for any fixed $R_0 > 0$.
\end{Prop}

\subsection{Kuranishi model for $\mathcal{F}$}Since $\mathscr{D}$ is elliptic, $\mathscr{D}^*\mathscr{D}$ and $\mathscr{D}\mathscr{D}^*$ share the same real discrete spectrum $\{\lambda_1, \lambda_2,\cdots\}$ away from zero. Each eigenspace associated to an eigenvalue $\lambda$ is finite dimensional, a consequence of ellipticity of $\mathscr{D}$. Denote by $\mathcal{C}^{\lambda_n}$  the (orthogonal) direct sum of all eigenspaces of $\mathscr{D}^*\mathscr{D}$ where the eigenvalues are strictly larger $\lambda_n$, and $\mathcal{C}_{\lambda_n}$ by the direct sum of all eigenspaces of $\mathscr{D}^*\mathscr{D}$ where the eigenvalues are lesser than or equal to $\lambda_n$. Similarly, define $\mathcal{R}^{\lambda_n}$ and $\mathcal{R}_{\lambda_n}$ for $\mathscr{D}\mathscr{D}^*$. Note that both $\mathcal{C}_{\lambda_n}$ and $\mathcal{R}_{\lambda_n}$ are finite dimensional, and 
$$\mathscr{D} : \mathcal{C} = \mathcal{C}_{\lambda_n}\oplus \mathcal{C}^{\lambda_n} \to \mathcal{R}_{\lambda_n}\oplus \mathcal{R}^{\lambda_n} = \mathcal{R}$$
respects the orthogonal decomposition. Furthermore, $\mathscr{D}$ is "$L^2-$norm-preserving".

\begin{Lemma}
The map $\mathscr{D} : \mathcal{C}^{\lambda_n} \to \mathcal{R}^{\lambda_n}$ is an isomorphism between Hilbert spaces. 
\end{Lemma}

\begin{proof}
We only need to show that $\mathscr{D}$ is bijective. First, let $v \in \mathcal{C}^{\lambda_n}$ such that $\mathscr{D} v = 0$. Without loss of generality, suppose that $v$ is an eigenvector associated to an eigenvalue $\lambda > \lambda_n$. Then $\mathscr{D}^* \mathscr{D} v = \lambda\,v$. This immediately implies that $\lambda \, v = 0$. And since $\lambda \neq 0$, $v = 0$. So we have $\mathscr{D}$ injective. For surjectivity, we let $w \in \mathcal{R}^{\lambda_n}$ and without loss of generality assume that $w$ is an eigenvector associated to an eigenvalue $\lambda > \lambda_n$. Then $\mathscr{D}\mathscr{D}^* w = \lambda\, w$, $(1/\lambda)\mathscr{D}^* w \in \mathcal{C}^{\lambda_n}$ and thus, we have $\mathscr{D}$ onto.
\end{proof}

Let $\pi^{\lambda_n}: \mathcal{R} \to \mathcal{R}^{\lambda_n}$ be the orthogonal projection. Lemma $3.5$ tells us that the restricted map $\mathscr{D}$ is an isomorphism so there is a well-defined bounded inverse $\mathscr{D}^{-1}: \mathcal{R}^{\lambda_n} \to \mathcal{C}^{\lambda_n}$. To get a finite dimensional approximation for $\mathcal{F}$, we use the global Kuranishi model. The idea is straightforward. We first define the following map $\phi_n : \mathcal{C} \to \mathcal{C}$ (cf. p.334, \cite{S14})
$$\phi_n := 1_{\mathcal{C}} + \mathscr{D}^{-1}\pi^{\lambda_n}\mathscr{Q}$$
We would like to show that for a fixed radius $R > 0$ and large enough $n$, $\phi_n$ is injective with a well-defined differentiable inverse. Then the map $f_n := (1-\pi^{\lambda_n})\mathcal{F}\phi^{-1}_{n}$ from $\mathcal{C}_{\lambda_n} \to \mathcal{R}_{\lambda_n}$ is a finite dimensional approximation for $\mathcal{F}$. To prove this, we need the following technical lemmas.

\begin{Lemma}[Lemma B.12, \cite{S14}]
Let $X, Y, Z$ be Banach spaces and $Q_n :  Y\to Z$ be a sequence of bounded linear operators such that
$$\lim_{n\to \infty} Q_n y = 0$$
for all $y \in Y$. Moreover, let $\{K_\alpha\}_{\alpha \in A}$ by a collection of bounded linear operators $K_\alpha : X \to Y$, indexed by a set $A$, such that the set 
$$B = \{K_{\alpha}x : \alpha \in A, x\in X, \norm{x}\leq 1\} \subset Y$$
has compact closure. Then
$$\lim_{n\to \infty} \sup_{\alpha} \norm{Q_n K_{\alpha}} = 0.$$ 
\end{Lemma}

\begin{Lemma}[Lemma B.2, \cite{S14}]
Let $X$ be a Banach space and $\psi : X \to X$ be a continuously differential map such that $\psi(0) = 0$ and $\norm{1 - d_x \psi} \leq \gamma$ for all $x \in X$ with $\norm{x} < R$, $\gamma < 1$ is some constant. Then the restriction of $\psi$ to $B(0, R)$ is injective, $\psi(B(0,R))$ is an open set, and $\psi^{-1} : \psi(B(0,R)) \to B(0, R)$ is continuously differentiable with $d_{y}\psi^{-1} = [d_{\psi^{-1}(y)}\psi]^{-1}$. Moreover, $B(0, R(1-\gamma)) \subset \psi(B(0,R)) \subset B(0, R(1+\gamma))$.
\end{Lemma}

\begin{Lemma}
Let $\phi_n v = u$. Then $\mathcal{F} v = 0$ if and only if $u \in \mathcal{C}_{\lambda_n}$ and $(1 - \pi^{\lambda_n})\mathcal{F} v = 0.$
\end{Lemma}

\begin{proof}
Suppose that $\mathcal{F}v = \mathscr{D}v +\mathscr{Q}v = 0$, this is equivalent to $\mathscr{Q}v = -\mathscr{D} v$. Then from $\phi_n v = u$, we obtain $v - \mathscr{D}^{-1} \pi^{\lambda_n} \mathscr{D} v = u$. Note that $\mathscr{D}$ and $\pi^{\lambda_n}$ commute. As a result, $v - \pi^{\lambda_n} v = u$. But this is equivalent to saying $u \in \mathcal{C}_{\lambda_n}$. Obviously, $(1-\pi^{\lambda_n})\mathcal{F} v = 0$. 

Conversely if $u \in \mathcal{C}_{\lambda_n}$ and $(1-\pi^{\lambda_n})\mathcal{F}v = 0$, then $u = v + \mathscr{D}^{-1}_1 \pi^{\lambda_n} (\mathcal{F} - \mathscr{D}) v = v + \mathscr{D}^{-1} \pi^{\lambda_n}\mathcal{F} v - \pi^{\lambda_n} v$. Apply $\mathscr{D}$ to both sides, we obtain $\mathscr{D} u = \mathscr{D} v + \mathcal{F} v -\pi^{\lambda_n} \mathscr{D} v = (1-\pi^{\lambda_n})\mathscr{D}v + \mathcal{F} v$. Since $u \in \mathcal{C}_{\lambda_n}$, $\mathscr{D} u \in \mathcal{R}_{\lambda_n}$. And note that $(1-\pi^{\lambda_n})\mathscr{D}v$ is clearly also in $\mathcal{R}_{\lambda_n}$. Hence $\mathcal{F} v \in \mathcal{R}_{\lambda_n} \cap \mathcal{R}^{\lambda_n}$. This is possible only if $\mathcal{F} v =0$.
\end{proof}

\begin{Prop}
Assume that $\mathcal{M}_g:= \mathcal{F}^{-1}(0)$ is compact. Then there exists an $C>0$ and $n$ large enough such that $f_n : \mathcal{C}_{\lambda_n} \to \mathcal{R}_{\lambda_n}$ defined above has no zero on $\{u \in \mathcal{C}_{\lambda_n}: \norm{u}_{L^2_{k}} = C\}$. 
\end{Prop}

\begin{proof}
With $\mathcal{M}_g$ being compact, there is an $R > 0$ such that if $\norm{v}_{L^2_k} \geq R$, we have $\mathcal{F} v \neq 0$. Note that $\phi_n(0) = 0$ and for any $v \in \mathcal{C}$, $d_{v}\phi_n = 1_{\mathcal{C}} +\mathscr{D}^{-1}\pi^{\lambda_n}d_{v}\mathscr{Q}$ so that  $d_{v}\phi_n - 1_{\mathcal{C}} = \mathscr{D}^{-1}\pi^{\lambda_n}d_{v}\mathscr{Q}$. Since $\mathscr{D}: \mathcal{C}^{\lambda_n} \to \mathcal{R}^{\lambda_n}$ is an isomorphism, we have
$$\norm{d_v\phi_n - 1_{\mathcal{C}}} = \norm{\pi^{\lambda_n}d_v\mathscr{Q}}.$$
Note that $\lim_{n\to \infty} \pi^{\lambda_n} w = 0$ for all $w \in \mathcal{R}$ and Proposition $3.4$ tells us that $\{d_v\mathscr{Q}\}_{v \in B(0,3R)}$ is a uniform family of compact operators. So, by Lemma $3.6$, we have 
$$\lim_{n \to \infty} \sup_{v \in B(0,3R)}\norm{\pi^{\lambda_n}d_v\mathscr{Q}} = 0.$$ 
This implies that for an $n$ large enough, say for all $n\geq n_0(R)$,
$$\norm{d_{v}\phi_n - 1_{\mathcal{C}}} \leq \frac{1}{2} \, \, \, \text{whenever } v \in B(0,3R).$$
Then by Lemma 3.7, $\phi_n$ restricted to $B(0,3R)$ is injective and has a well-defined continuously differentiable inverse. Furthermore, $\phi_n(B(0,3R))$ is open and 
$$\phi_n(B(0,R)) \subset B(0, 3R/2) \subset \phi_n(B(0,3R)) \subset B(0, 9R/2).$$
Suppose that there is $u\in \partial B(0,3R/2) \cap \mathcal{C}_{\lambda_n}$ such that $f_n u =0$. Let $v = \phi^{-1}_n u$ or $u = \phi_n v$. By Lemma $3.8$, $\mathcal{F} v = 0$. But since $\norm{v}_{L^2_k} = \norm{\phi^{-1}_n u}_{L^2_k} \geq R$, this leads to a contradiction. Therefore, for any $u \in \mathcal{C}_{\lambda_n}$ where $\norm{u}_{L^2_k} = C$ and $C > 3R/2$ such that $B(0,C) \subset \phi_n(B(0,3R))$, $f_n u \neq 0$.
\end{proof}

\begin{Rem}
If $\mathcal{M}_g$ is empty, one can pick any $R > 0$ in the beginning of the proof of Proposition $3.9$. In other words, one can consider any finite dimensional approximation of $\mathcal{F}$ and it would always produce a $Pin(2)-$map between spheres.
\end{Rem}

\begin{Rem}
The finite dimensional approximation construction of $\mathcal{F}$ above is similar to Furuta's in his paper \cite{F01}.
\end{Rem}

\begin{Rem}
There is another way to obtain finite dimensional approximation for $\mathcal{F}$ (cf. p. 335, \cite{S14}) and one still gets the same result in Proposition 3.9. Alternatively, we define $g_n := (1- \pi^{\lambda_n})(\mathscr{D} + \mathscr{Q}\phi_n^{-1})$. With the same hypothesis in Proposition 3.9 and a specified ball $B(0,R)$ in $\mathcal{C}$ such that $\mathcal{F}$ is never zero outside such ball indicated in the beginning of the above proof, $\phi_n$ restricted to $B(0,3R)$ is still injective and has a well-defined continuously differential inverse for large enough $n$. Then one notes that
\begin{align}
&\mathscr{D}\phi_n -\mathscr{D} - \pi^{\lambda_n}\mathcal{F} + \pi^{\lambda_n}\mathscr{D} = \mathscr{D}\phi_n - \mathscr{D} - \pi^{\lambda_n}(\mathscr{D} + \mathscr{Q}) + \pi^{\lambda_n}\mathscr{D}\nonumber\\
&=\mathscr{D}\phi_n - \mathscr{D} - \pi^{\lambda_n}\mathscr{Q} = \mathscr{D}(1_{\mathcal{C}} + \mathscr{D}^{-1}\pi^{\lambda_n}\mathscr{Q})-\mathscr{D} -\pi^{\lambda_n}\mathscr{Q} = 0.\nonumber
\end{align}
Apply $\phi_n^{-1}$ on the right of the above equations, and we get exactly that $\mathscr{D} - \mathscr{D}\phi_n^{-1}-\pi^{\lambda_n}\mathcal{F}\phi_n^{-1} + \pi^{\lambda_n}\mathscr{D} \phi_n^{-1} = 0$. Therefore,
\begin{align}
g_n &= (1- \pi^{\lambda_n})(\mathscr{D} + \mathcal{F} \phi_n^{-1}  - \mathscr{D}\phi_n^{-1})\nonumber \\
& = \mathscr{D} + \mathcal{F} \phi_n^{-1} - \mathscr{D}\phi_n^{-1} - \pi^{\lambda_n}\mathscr{D} - \pi^{\lambda_n} \mathcal{F}\phi_n^{-1} + \pi^{\lambda_n}\mathscr{D}\phi_n^{-1}\nonumber\\
& = \mathcal{F}\phi_n^{-1} - \pi^{\lambda_n}\mathscr{D}.\nonumber
\end{align}
Hence, $\mathcal{F}\phi_n^{-1} = \pi^{\lambda_n}\mathscr{D} + g_n$. So if $u \in \mathcal{C}_{\lambda_n}$ and $\norm{u}_{L^2_k} \geq 3R/2$, then $\norm{\phi_n^{-1}u}_{L^2_k} \geq R$. This implies that $\mathcal{F}\phi_n^{-1}u\neq 0$. Because $\mathscr{D} u \in \mathcal{R}_{\lambda_n}$, so $\pi^{\lambda_n}\mathscr{D} u = 0$. As a result, $g_n u \neq 0$ also.
\end{Rem}

\begin{Rem}
The hypothesis that there is a ball of a certain radius centered at $0$ in $\mathcal{C}$ such that $\mathcal{F} \neq 0$ outside such ball is not needed to obtain a finite dimensional approximation for $\mathcal{F}$. This condition only ensures that once we have a finite dimensional approximation $f_n$ of $\mathcal{F}$, $f_n \neq 0$ also on a specified sphere whose existence depends on the aforementioned ball.
\end{Rem}

Note that the map $f_n : \mathcal{C}_{\lambda_n} \to \mathcal{R}_{\lambda_n}$ in Proposition 3.9 is also $Pin(2)-$equivariant by construction. Recall that we have shown $\mathscr{D}$ to be $Pin(2)-$equivariant. Thus, $\mathcal{C}_{\lambda_n}$ and $\mathcal{R}_{\lambda_n}$ are finite dimensional representations of $Pin(2)$. 

\begin{Lemma}[cf. Lemma 4.1, \cite{F01}]
There exists positive integers $t,s,r,q$ such that
$$ \mathcal{C}_{\lambda_n} = \mathbb{H}^{t}\oplus \mathbb{R}^{s}; \, \, \, \, \, \mathcal{R}_{\lambda_n} = \mathbb{H}^{r} \oplus \mathbb{R}^{q}.$$
\end{Lemma}

\begin{proof}
Let $\mathcal{C}^1_{\lambda}$ be the direct sum of all eigenspaces of  $Q^{-}Q^{+}$ where the eigenvalues are lesser than or equal to $\lambda$. Define similarly $\mathcal{R}^1_{\lambda}$ for $Q^+ Q^-$, $\mathcal{C}^2_{\lambda}$ and $\mathcal{R}^2_{\lambda}$ for $d^+\oplus d^*$. We have $\mathcal{C}_{\lambda}=\mathcal{C}^1_{\lambda}\oplus \mathcal{C}^2_{\lambda}$ and $\mathcal{R}_{\lambda} = \mathcal{R}^1_{\lambda} \oplus \mathcal{R}^2_{\lambda}$.

Note that fiber-wise, $(\mathfrak{s}^+_{3/2})_p \cong \mathbb{H}^3$. This is because if $\psi_p$ is a $3/2-$spinor vector, then $\psi_p$, $i\psi_p$ and $j\psi_p$ are all linearly independent. Since $\mathcal{C}^1_{\lambda}$ is a finite dimensional subspace of $L^2_{k}(\mathfrak{s}^+_{3/2})$, there are finite number of points $p_1,\cdots, p_\alpha \in X$ such that the evaluation map
$$ev : \mathcal{C}^1_{\lambda} \to \bigoplus_{\ell=1}^{\alpha}(\mathfrak{s}^+_{3/2})_{p_{\ell}}\cong \mathbb{H}^{A}$$
is injective and $Pin(2)-$equivariant for some positive integer $A$. But $\mathbb{H}$ is an irreducible representation of $Pin(2)$. Thus, $\mathcal{C}^1_{\lambda} \cong \mathbb{H}^t$ for some positive integer $t$. Argue similarly for $\mathcal{R}^1_{\lambda}$, $\mathcal{C}^2_{\lambda}$ and $\mathcal{R}^2_{\lambda}$.
\end{proof}

Since $\mathscr{D}: \mathcal{C}^{\lambda_n} \to \mathcal{R}^{\lambda_n}$ is an isomorphism by Lemma 3.5, $2\,index_{\CN}\, Q^+ = 4(t-r)$ and $index\, (d^+\oplus d^*) + 1= s - q$. The $+ 1$ is needed because $\mathcal{R}$ does not contain constant functions. We know that $index\, (d^+ \oplus d^*) = -b_2^+(X) -1$; and Proposition 2.2 tells us that $index\, Q^+ = 19\,\sigma(X)/8$. Hence, $t = r + 19\,\sigma(X)/16$ and $q = s+b^+_2(X)$, where $b^+_2(X) = \text{dim}\, H^+(X)$. Let $k = 19\,\sigma(X)/16$ and $m = b^+_2(X)$.  As a result,
$$\mathcal{C}_{\lambda_n}  = \mathbb{H}^{r+k} \oplus \mathbb{R}^{s}; \, \, \, \, \, \mathcal{R}_{\lambda_n} = \mathbb{H}^{r} \oplus \mathbb{R}^{s+m}.$$
Now the complexification of $f_n$ would still be $Pin(2)-$equivariant. Denote the complexified $f_n$ by itself again for convenience. We have a smooth equivariant map
$$f_n : V:=\mathbb{H}^{2r+2k} \oplus \mathbb{C}^{s} \to \mathbb{H}^{2r} \oplus \CN^{s+m}:=W.$$
The above map also induces a smooth $Pin(2)-$equivariant map $f: \mathcal{B}V/\mathcal{S}V \to \mathcal{B}W/\mathcal{S}W$, where $\mathcal{B}$ and $\mathcal{S}$ denotes the closed unit ball and closed unit sphere of a vector space. In the next section, using equivariant $K-$theory, we will show that if such a map exists, then we must have $m \geq 2k +1$.

\section{Proof of Theorem 1.2}

\subsection{Equivariant $K-$theory} So far, we have seen that if the moduli space of solutions $\mathcal{M}_g$ is compact, then there exists a $Pin(2)-$equivariant map $f : \mathcal{B}V/\mathcal{S}V \to \mathcal{B}W/\mathcal{B}W$ between spheres, where $V$ and $W$ are $Pin(2)-$representations constructed in the previous section. If the following proposition holds, then we immediately obtain Theorem 1.2.

\begin{Prop}[Proposition 9.73, \cite{S14}, cf. Proposition 5.1,\cite{F01}]
If there exists a $Pin(2)-$equivariant map $f :\mathcal{B}V/\mathcal{S}V \to \mathcal{B}W/\mathcal{S}W$, where $V = \mathbb{H}^{2r+2k}\oplus \CN^s:=V_0\oplus V_1$ and $W = \mathbb{H}^{2r}\oplus\CN^{s+m}:=W_0\oplus W_1$ are $Pin(2)-$representations, then either $k=0$ or $m\geq 2k+1$.
\end{Prop}

\noindent
\textit{Proof of Theorem 1.2 assuming Proposition 4.1.} Indeed, suppose that Proposition 4.1 is true for now, apply it to the setting where $f$ is the induced map of a finite dimensional approximation for $\mathcal{F}$ where $\mathcal{F}^{-1}(0) = \mathcal{M}_g$ is compact, and $m = b_2^+(X)$, $k=19\sigma(X)/16$. Since $X$ is a closed simply connected smooth spin $4-$manifold whose intersection form is indefinite, obviously $b_2(X) \geq 2$. The inequality in Theorem 1.2 satisfies vacuously when $k=0$. Otherwise, $m \geq 2k+1$. Equivalently, $b_2^+(X) \geq 19\,\sigma(X)/8 +1$. Scaling the inequality by 2, we have
$$b_2(X) + \sigma(X)=2\,b^+_2(X) = 2m  \geq 4k + 2 = \frac{19}{4}\,\sigma(X) + 2.$$
The above inequality implies that $b_2(X) \geq 15\,\sigma(X)/4 + 2$ as claimed. \qed

Therefore, what remains for us to do in this section is to prove Proposition 4.1. To do that, we need to recall some facts about equivariant $K-$theory for $Pin(2)$. The results and definitions that are about to be listed exist in a variety of places in the literature; for example, see pg. 337 in \cite{S14}. We summarize them here for the sake of self-containment.

Let $M$ be any compact Hausdorff space and $G$ is any compact Lie group acting on $M$. An $G-$equivariant complex vector bundle is a complex vector bundle $\pi: E\to M$ that carries a $G-$action such that $\pi$ is equivariant. The group $K_G(M)$ is the Grothendieck group of the semigroup of equivalence classes of complex $G-$vector bundles over $M$. For two complex $G-$vector bundles $E \to M$ and $F \to M$, we write $E \ominus F$ as its equivalence class in $K_G(M)$, meaning that $E\ominus F \sim E'\ominus F'$ if and only if  $E\oplus F'$ is stably isomorphic to $F \oplus E'$ by some $G-$equivariant vector bundle over $X$. Note that $K_G(\star)$ is exactly the Grothendieck ring of the set containing equivalent classes of finite dimensional unitary representations of $G$, i.e, the representation ring $\mathcal{R}(G)$. The functor $K_G$ is a homotopy invariant; and hence, any contractible space $M$ has $K_G(M) \cong \mathcal{R}(G)$. The functor $K_G$ is also contravariant, and when $M$ has no $G-$action, $K_G(M) := K(M) \otimes \mathcal{R}(G)$.

If $N \subset M$ is also compact and preserved by the $G$ action, the inclusion $\star \to M/N$ induces a group homomorphism $K_G(M/N) \to \mathcal{R}(G)$. The relative $K_G-$group $K_G(M,N)$ is defined to be the kernel of aforementioned group homomorphism. Elements of $K_G(M,N)$ are represented by $E \ominus_{\varphi} F$ such that $\varphi: E|_{N} \to F|_{N}$ is an isomorphism. 

Let $V$ be a finite dimensional unitary representation of $G$. Naturally, $\mathcal{B}V$ and $\mathcal{S}V$ are compact and have $G-$actions. So we set $M:= \mathcal{B}V$ and $N:=\mathcal{S}V$. Then, define the \textit{equivariant Thom class} $\tau_V \in K_G(\mathcal{B}V,\mathcal{S}V)$ by
$$\tau_V := \Lambda^{0,\text{even}}V^*\ominus_{\varphi} \Lambda^{0,\text{odd}} V^*.$$
The following theorem is an important and deep result that tells us $K_G(\mathcal{B}V, \mathcal{S}V)$ is a module over $\mathcal{R}(G)$, and as a module, it is generated by $\tau_V$. The proof of it can be found in \cite{A67}. 

\begin{Th}[\textbf{Bott}]
Suppose $V$ is a finite dimensional unitary representation of $G$. Then $K_G(\mathcal{B}V, \mathcal{S}V)$ is naturally isomorphic to $\mathcal{R}(G)$ via the homomorphism
$$\mathcal{R}(G) \to K_G(\mathcal{B}V, \mathcal{S}V), \, \, \, \, \, \rho \mapsto \rho \otimes \tau_V.$$
\end{Th}

The above theorem lets us define a notion called $K_G-$theoretic degree. Let $V$ and $W$ are finite dimensional unitary representations of $G$. Suppose we have a smooth equivariant map $f: \mathcal{B}V/\mathcal{S}V \to \mathcal{B}W/\mathcal{S}W$ so that one has an induced map $f^*: K_G(\mathcal{B}W/\mathcal{S}W) \to K_G(\mathcal{B}V/\mathcal{S}V)$. Since $f^*\tau_W \in K_G(\mathcal{B}V, \mathcal{S}V)$, by Theorem 4.1, there must be a unique element $a_f \in \mathcal{R}(G)$ such that $f^*\tau_W = a_f \otimes \tau_V$. This element $a_f \in \mathcal{R}(G)$ associated to a smooth map $f$ between spheres is the $K_G-$theoretic degree of $f$. Note that $a_f$ could be formal difference between two unitary representations. 

Beside Theorem 4.2, we also need to know exactly what the representation ring of $Pin(2)$ looks like before proving Proposition 4.1.

\begin{Th}[pg. 338, \cite{S14}]
The representation ring of $Pin(2)$ is isomorphic to the quotient ring over $\ZN$ (also an $\ZN-$module)
$$\mathcal{R}(Pin(2)) \cong \frac{\ZN[d,h]}{\la d^2 -1, dh-h\ra}.$$
In particular, $d$ is associated with the unitary representation of $Pin(2)$ over $\CN$ where $j \mapsto -1$ and $e^{it} \mapsto 1$; and $h$ is associated with the usual representation of the group on $\mathbb{H}$ (See subsection 2.3). In particular, every $1-$dimensional unitary representation is either associated with $1$ or $d$; and every $2-$dimensional unitary representation is of the form
$$ j \mapsto \begin{pmatrix} 0 & -1 \\ 1 & 0 \end{pmatrix}, \, \, \, \, \, e^{it} \mapsto \begin{pmatrix} e^{int} & 0 \\ 0 & e^{-int} \end{pmatrix}.$$
\end{Th}

\begin{Lemma}
For $d, h \in \mathcal{R}(Pin(2))$ in Theorem 4.3, we have $\lambda_d = 1 - d$ and $\lambda_h = 2 - h$. Furthermore as elements in $\mathcal{R}(Pin(2))$, $(1-d)^{\alpha} = 2^{\alpha - 1}(1-d)$ and $(2-h)^{\beta}(1-d) = 2^{\beta}(1-d)$ for any integers $\alpha \geq 1, \beta \geq 0$.
\end{Lemma}

\begin{proof}
$d$ is associated to the unitary representation of $Pin(2)$ on $\CN$. Note that $\Lambda^{0,0} \CN$ is the base field which corresponds to the trivial representation. And $\Lambda^{0,1}\CN$ (which is generated by $dz$) is canonically isomorphic to $\CN$; so, it corresponds to $d$. As a result, $\lambda_d = \Lambda^{0,0}\CN - \Lambda^{0,1}\CN = 1 - d$.

We identify $\mathbb{H} = \CN^2$ as in subsection 2.3. Since $\Lambda^{0,0} \mathbb{C}^2$ is the base field which also corresponds to the trivial representation and $\Lambda^{0,2} \CN^2$ (which is generated by $d\bar{z_1}\wedge d\bar{z_2}$) is isomorphic to $\CN$, $\Lambda^{0,0}\CN^2 \oplus \Lambda^{0,2} \CN^2 = 1 + 1 =2$ in $\mathcal{R}(Pin(2))$. On the other hand, $\Lambda^{0,1}\CN^2$ (which is generated by $d\bar{z_1}$, $d\bar{z_2}$) is isomorphic to $\CN^2 = \mathbb{H}$. Hence, in $\mathcal{R}(Pin(2))$, $\Lambda^{0,1}\CN^2 = h$. Therefore, $\lambda_h = 2 -h$.

Now we prove the two identities by induction. Recall that in $\mathcal{R}(Pin(2))$, $(1-d)^2 = 1 - 2d + d^2 = 2(1-d)-$this is our base case. Suppose that the $(1-d)^{\alpha} = 2^{\alpha - 1}(1-d)$ for all $\alpha$ up to $n\geq 2$. Then one sees that
$$(1-d)^{n+1}  = (1-d) \, 2^{n-1}(1-d) = 2^{n-1} (1-d)^2 = 2^{n}(1-d).$$
This completes the induction. Similarly, since $dh - h =0$ in the representation ring, we have $(2-h)(1-d) = 2 - 2d - h + dh = 2(1-d)-$this is the base case for $\beta = 1$. Assume that the identity holds for all $\beta$ up to $n \geq 1$. Then we see that
$$(2-h)^{n+1}(1-d) = (2-h)\, 2^{n}(1-d)  = 2^{n} \, 2 (1-d) = 2^{n+1}(1-d).$$
The induction is concluded. We have our result as claimed.
\end{proof}

\subsection{Proof of Proposition 4.1}With $V, W, f$ set up as in the end of Section 3, since we have a smooth equivariant map $\mathcal{B}W \to \mathcal{B}W/\mathcal{S}W$, this induces a homomorphism that is also $\mathcal{R}(G)-$linear $K_B(\mathcal{B}W, \mathcal{S}V) \to K_G(\mathcal{B}W)$ taking $\tau_W \mapsto \lambda_W$. Similarly, for $V$. Note that both $K_G(\mathcal{B}W)$ and $K_G(\mathcal{B}V)$ are naturally isomorphic to $\mathcal{R}(G)$. So, we have the following commutative diagram

$$\begin{tikzcd}
\tau_W\in K_G(\mathcal{B}W,\mathcal{S}W) \arrow{d} \arrow{r}{f^*}   & K_G(\mathcal{B}V,\mathcal{S}V)\ni a_f \otimes \tau_V \arrow{d} \\
\lambda_W \in K_G(\mathcal{B}W)  \arrow{d} \arrow{r}{f^*}   & K_{G}(\mathcal{B}V)\ni a_f \otimes \lambda_V \arrow{d}  \\
\mathcal{R}(G) \arrow{r}{\cong} & \mathcal{R}(G)
\end{tikzcd}$$
Thus, $\lambda_W = a_f \lambda_V$. Now because $\lambda$ splits direct sum of representations into product of  their $\lambda$'s in the representation ring, we have
$$ \lambda_V = (2-h)^{2r + 2k}(1-d)^s, \, \, \, \, \, \lambda_W = (2-h)^{2r}(1-d)^{s+m}.$$
By Lemma 4.4, we see that in $\mathcal{R}(G)$,
\begin{align}
\lambda_V = (2-h)^{2r+2k}\,2^{s-1}(1-d) = 2^{s-1}\,2^{2r + 2k}(1-d) = 2^{2r + 2k + s -1}(1-d).
\end{align} 
Similarly, 
\begin{align}
\lambda_W = (2-h)^{2r}\,2^{m-1}(1-d) = 2^{m-1}\,2^{2r}(1-d) = 2^{2r+s+m-1}(1-d).
\end{align}
From $(4.1)$ and $(4.2)$, we obtain
\begin{align}
2^{2r+s+m-1}(1-d) = 2^{2r+2k+s-1}\,a_f (1-d).
\end{align}
Since $\text{tr}\,((1-d)(j)) = \text{tr}\,(1(j)) - \text{tr}\,(d(j)) = 1 -(-1) =2$, when apply traces evaluated at $j$ to $(4.3)$, we have
\begin{align}
2^{2r+s+m} = 2^{2r+2k+s}\,\text{tr}\,(a_f(j)) 
\end{align}
Now as element in $\mathcal{R}(G)$, by Theorem 4.3, we can write $a_f$ as
$$ a_f = \alpha_0\, 1 + \alpha_1 \, d + \alpha_{\ell} \, h^{\ell} + \dots + \alpha_2 \, h.$$
Here $\alpha_0,\dots, \alpha_{\ell}$ are all integers. When evaluating the trace of $a_f(j)$, we have
$$\text{tr}\,(a_f(j)) = \alpha_0 - \alpha_1 + \alpha_{\ell}\,\text{tr}\,(h^{\ell}(j)) + \dots  + \alpha_2\,\text{tr}\,(h(j)).$$
By Theorem 4.3 again, $\text{tr}\,(h(j)) = 0$ and thus the trace of any power of $h$ evaluated at $j$ is also zero, it turns out that $\text{tr}\,(a_f(j)) = \alpha_0 - \alpha_1 \in \ZN$. Combine with $(4.4)$, it implies that $m \geq 2k$. So, $\text{tr}\,(a_f(j)) \geq 1$. Thus either $k=0$ or for $m \geq 1$, $\text{tr}\,(a_f(j))$ cannot be $1$. Assume otherwise, evaluating the trace of $a_f$ at $i$ gives us
$$\text{tr}\,(a_f(i))  = \alpha_0 + \alpha_1 + \alpha_{\ell}\,\text{tr}\,(h^{\ell}(i)) + \dots  + \alpha_2\,\text{tr}\,(h(i)).$$
By Theorem 4.3, the trace of any power of $h$ evaluated at $i$ is exactly zero; and by Lemma 9.43 in \cite{S14} (cf. \cite{F01}), $\text{tr}\,(a_f(i)) = 0$. As a result, $\alpha_0 + \alpha_ 1 =0$. But it is impossible to find integers  $\alpha_0$, $\alpha_1$ that satisfy $\alpha_0 +  \alpha_1 = 0$ and $\alpha_0 - \alpha_1 =1$ simultaneously. Therefore, $\text{tr}\,(a_f(j)) \geq 2$. Going back to $(4.4)$, we obtain $m \geq 2k+1$ as claimed. \qed

\section{Conclusion}
Theorem 1.2 and Corollary 1.3 give a sufficient condition for a closed simply connected smooth spin $4-$manifold to always have non-trivial solutions to its associated RS-SW equations (2.14). Then necessarily, the moduli space of solutions is always non-compact. If we categorize the solutions based on the action of gauge group $\mathcal{G}$, there are \textit{reducible solutions} and \textit{irreducible solutions} just like in Seiberg-Witten theory. Appropriate perturbation of (2.14) would help us avoid reducible solutions altogether, while producing the same sufficient condition for non-compactness of the moduli space of solutions. As a result, a transversality theorem will ensure that away from the singularities, the moduli space $\mathcal{M}_g$ of solutions to the appropriately perturbed version of (2.14) is an actual (compact or not) manifold.

The expected dimension of $\mathcal{M}_g$ can be computed by first deforming the RS-SW equations. And the corresponding deformation of the elliptic complex associated to the RS-SW equations with Atiyah-Singer index theorem would give us a formula for the virtual dimension of $\mathcal{M}_g$, $d =  19\,\sigma(X)/4-1-b^+_2(X)$. Note that if $X$ has non-positve signature, then the moduli space of solution would be empty (and compact trivially). Hence, the $15/4-$bound in Theorem 1.2 is satisfied vacuously. Therefore, focusing strictly on positive signature case, it would be interesting to know the answer to the following question

\begin{?}
Is there a closed simply connected smooth spin $4-$manifold $X$ with positive signature that possesses a Riemannian metric $g$ such that $\mathcal{M}_g$ is compact non-trivially?
\end{?}

\begin{Rem}
Necessarily, such a manifold $X$ must satisfy $15\,\sigma(X)/4+2 \leq b_2(X) \leq 17\,\sigma(X)/2-2$, thus satisfies the 11/8th conjecture. By Freedman's theorem \cite{F82}, $X \cong \#_{k} \overline{K3} \#_{\ell} S^2 \times S^2$ up to homeomorphism, where $k, \ell > 0$. If $X$ happens to be symplectic, then it cannot be diffeomorphic to  $\#_{k} \overline{K3} \#_{\ell} S^2 \times S^2$ by Taubes' vanishing/non-vanishing theorem \cite{HT99}. Note that within the slopes $[15/4,17/2]$ of $b_2$, Park \cite{P02} and Akhmedov-Park \cite{AP10} constructed examples of closed simply connected symplectic spin $4-$manifold with positive signature. It would be interesting to find out if any of those manifolds has a compact associated moduli space $\mathcal{M}_g$.
\end{Rem}

Let $n>0$ be any integer. Consider a closed simply connected \textit{smoothable topological} $4-$manifold $X$ whose intersection form is realized by $Q_X = 2n\,E8 \oplus 3n\,H$. Consequently, $b_2(X) = 22n$ and $\sigma(X) = 16n$. Up to homeomorphism, $X \cong \#_{n}\overline{K3}$. Already, one can see that as a consequence of Theorem 1.2, the moduli space $\mathcal{M}_g$ associated to any given smooth structure on $X$ is never compact. If an invariant can be constructed from $\mathcal{M}_g$, potentially it can identify exotic smooth structure on $X$. In contrast to Seiberg-Witten theory, Taubes' vanishing theorem \cite{HT99} tells us that the Seiberg-Witten invariant vanishes identically for such manifolds. Finally, it is also interesting to find out the answer to the following question

\begin{?}
Let $X$ be a closed simply connected topological $4-$manifold $X$ whose intersection form is $2n\,E8 \oplus 3n\,H$. Is there an exotic structure on $X$ such that $X$ splits off smoothly an $S^2\times S^2$?
\end{?}

For $n = 1$, the answer is a definitive no from Furuta's 10/8th theorem \cite{F01}. Indeed, if $X$ is diffeomorphic to $N \# (S^2 \times S^2)$, then $b_2(N) - 5\,\sigma(N)/4 - 2 = 20-5\cdot 16/4 - 2 <0$. This implies that $N$ cannot be smoothable. When $n \geq 2$, Furuta's theorem does not give an apparent contradiction. However in this case, if $X$ has an exotic smooth structure such that $X \cong N \# (S^2\times S^2)$, then the moduli space $\mathcal{M}_g$ associated to $N$ is non-compact for any $g$. This is a consequence of Theorem 1.2, $b_2(N) - 15\,\sigma(N)/4 -2 = (22n-2)-15\cdot 16n/4 - 2 = -38n-4<0$. Note that such exotic $X$ cannot be symplectic by Taubes' vanishing/non-vanishing theorem \cite{HT99}.

We hope to address the above questions in our future work.

\bibliographystyle{plain}
\bibliography{main.bib}

\begin{thebibliography}{10}

\bibitem{AP10}
Anar Akhmedov and B.~Doug Park.
\newblock Geography of simply connected spin symplectic 4-manifolds.
\newblock {\em Math. Res. Lett.}, 17(3):483--492, 2010.

\bibitem{A67}
M.~F. Atiyah.
\newblock {\em {$K$}-theory}.
\newblock W. A. Benjamin, Inc., New York-Amsterdam, 1967.
\newblock Lecture notes by D. W. Anderson.

\bibitem{AS68}
M.~F. Atiyah and I.~M. Singer.
\newblock The index of elliptic operators. {III}.
\newblock {\em Ann. of Math. (2)}, 87:546--604, 1968.

\bibitem{BM21}
Christian B\"{a}r and Rafe Mazzeo.
\newblock Manifolds with many {R}arita-{S}chwinger fields.
\newblock {\em Comm. Math. Phys.}, 384(1):533--548, 2021.

\bibitem{BH02}
Thomas Branson and Oussama Hijazi.
\newblock {B}ochner-{W}eitzenb\"{o}ck formulas associated with the
  {R}arita-{S}chwinger operator.
\newblock {\em International Journal of Mathematics}, 13(2):137--182, 2002.

\bibitem{F82}
Michael~Hartley Freedman.
\newblock The topology of four-dimensional manifolds.
\newblock {\em J. Differential Geometry}, 17(3):357--453, 1982.

\bibitem{F01}
M.~Furuta.
\newblock Monopole equation and the {$\frac{11}8$}-conjecture.
\newblock {\em Math. Res. Lett.}, 8(3):279--291, 2001.

\bibitem{HW15}
Andriy Haydys and Thomas Walpuski.
\newblock A compactness theorem for the {S}eiberg-{W}itten equation with
  multiple spinors in dimension three.
\newblock {\em Geom. Funct. Anal.}, 25(6):1799--1821, 2015.

\bibitem{HS19}
Yasushi Homma and Uwe Semmelmann.
\newblock The kernel of the {R}arita-{S}chwinger operator on {R}iemannian spin
  manifolds.
\newblock {\em Comm. Math. Phys.}, 370(3):853--871, 2019.

\bibitem{HT99}
Michael Hutchings and Clifford~Henry Taubes.
\newblock An introduction to the {S}eiberg-{W}itten equations on symplectic
  manifolds.
\newblock In {\em Symplectic geometry and topology ({P}ark {C}ity, {UT},
  1997)}, volume~7 of {\em IAS/Park City Math. Ser.}, pages 103--142. Amer.
  Math. Soc., Providence, RI, 1999.

\bibitem{LM89}
H.~Blaine Lawson, Jr. and Marie-Louise Michelsohn.
\newblock {\em Spin geometry}, volume~38 of {\em Princeton Mathematical
  Series}.
\newblock Princeton University Press, Princeton, NJ, 1989.

\bibitem{M12}
Rafe Mazzeo, Jan Swoboda, Hartmut Wei{\ss}, and Frederik Witt.
\newblock Limiting configurations for solutions of {Hitchin{\textquoteright}s}
  equation.
\newblock {\em S\'eminaire de th\'eorie spectrale et g\'eom\'etrie},
  31:91--116, 2012-2014.

\bibitem{P02}
Jongil Park.
\newblock The geography of {S}pin symplectic 4-manifolds.
\newblock {\em Math. Z.}, 240(2):405--421, 2002.

\bibitem{P22}
Gregory~J. Parker.
\newblock Concentrating local solutions of the two-spinor seiberg-witten
  equations on 3-manifolds, 2022.

\bibitem{P23}
Gregory~J. Parker.
\newblock Deformations of {Z/2}-harmonic spinors on 3-manifolds, 2023.

\bibitem{RS41}
William Rarita and Julian Schwinger.
\newblock On a theory of particles with half-integral spin.
\newblock {\em Phys. Rev.}, 60:61--61, Jul 1941.

\bibitem{S14}
Dietmar~A. Salamon.
\newblock Spin geometry and seiberg-witten invariants.
\newblock 2014.

\bibitem{T82}
Clifford~Henry Taubes.
\newblock Self-dual {Y}ang-{M}ills connections on non-self-dual
  {$4$}-manifolds.
\newblock {\em J. Differential Geometry}, 17(1):139--170, 1982.

\bibitem{T84}
Clifford~Henry Taubes.
\newblock Self-dual connections on {$4$}-manifolds with indefinite intersection
  matrix.
\newblock {\em J. Differential Geom.}, 19(2):517--560, 1984.

\bibitem{T12}
Clifford~Henry Taubes.
\newblock {PSL}(2;c) connections on 3-manifolds with {L}2 bounds on curvature,
  2012.

\bibitem{T13}
Clifford~Henry Taubes.
\newblock Compactness theorems for {SL}(2;c) generalizations of the
  4-dimensional anti-self dual equations, 2013.

\bibitem{T14}
Clifford~Henry Taubes.
\newblock The zero loci of {Z}/2 harmonic spinors in dimension 2, 3 and 4,
  2014.

\bibitem{T16}
Clifford~Henry Taubes.
\newblock On the behavior of sequences of solutions to {U}(1) seiberg-witten
  systems in dimension 4, 2016.

\bibitem{WZ21}
Thomas Walpuski and Boyu Zhang.
\newblock On the compactness problem for a family of generalized
  {S}eiberg-{W}itten equations in dimension 3.
\newblock {\em Duke Math. J.}, 170(17):3891--3934, 2021.

\bibitem{W91}
McKenzie~Y. Wang.
\newblock Preserving parallel spinors under metric deformations.
\newblock {\em Indiana Univ. Math. J.}, 40(3):815--844, 1991.

\end{thebibliography}

\end{document}